\documentclass[11pt]{article}
\usepackage[margin=1.35in]{geometry}
\usepackage{amsmath}
\usepackage{amsthm}
\usepackage{amssymb}
\usepackage{amsfonts}
\usepackage{bbm} 
\usepackage{cite}
\usepackage{tikz}
 \usepackage{upgreek}
\usepackage{enumitem}
\usepackage{subfig}
\usepackage[labelfont=bf]{caption}
\usepackage{lmodern}
\usepackage{algorithm,algorithmic}
\usepackage{microtype}
\usepackage[bottom]{footmisc}
\usepackage{tikz}
\usetikzlibrary{shapes,arrows,graphs,graphs.standard,shapes.misc}
\usetikzlibrary{matrix,decorations.pathreplacing, calc, positioning,fit}
\usetikzlibrary{shapes.geometric}

\newtheorem{definition}{Definition}
\newtheorem{subdefinition}{Definition}[definition]

\newtheorem{theorem}{Theorem}
\newtheorem{lemma}{Lemma}
\newtheorem{corollary}[theorem]{Corollary}
\newtheorem{proposition}[theorem]{Proposition}

\newtheorem{example}{Example}

\makeatletter
\xdef\@endgadget#1{{\unskip\nobreak\hfil\penalty50\hskip1em\hbox{}\nobreak\hfil#1\parfillskip=0pt\finalhyphendemerits=0\par}}
\newcommand\@Endofsymbol{$\triangledown$}
\newcommand\Endofremark{\@endgadget{\@Endofsymbol}}
\makeatother

 \global\long\def\R{\mathbb{R}}
 
 \global\long\def\ker{\mathrm{Ker}}

\newcommand{\Z}{\mathbb{Z}}

\DeclareMathOperator{\rk}{rk}

\DeclareMathOperator{\rs}{rsl}
\DeclareMathOperator{\srs}{s-rsl}
\DeclareMathOperator{\mrs}{(s-)rsl}

\DeclareMathOperator{\diag}{diag}

\newcommand{\cS}{\mathcal{S}}

\newcommand{\bfl}{\bar{F}_\ell}
\newcommand{\fl}{{F}_\ell}

\makeatletter
\def\blfootnote{\xdef\@thefnmark{}\@footnotetext}
\makeatother

\newcounter{cAss}
\newcounter{cAssSaved}
\newcommand\Ass[1]{\ensuremath{\boldsymbol{\mathcal A}_{\text{\hspace{0.75pt}\bf#1}}}}
\newlength\asswidth
 {\begin{list}{\Ass{\arabic{cAss}}:}{%
   \setcounter{cAssSaved}{\thecAss}\usecounter{cAss}\setcounter{cAss}{\thecAssSaved}
   \setlength\topsep{.667ex plus 2pt minus 2pt}
   \setlength\itemsep\topsep
   \settowidth\asswidth{\Ass{\arabic{cAss}}:}
   \addtolength\asswidth\labelsep
   \setlength\leftmargin\asswidth
   \setlength\labelwidth\asswidth}}%
 {\end{list}}

\usepackage[color=yellow, textsize=small, textwidth=\marginparwidth, draft]{todonotes}   

\usepackage{xcolor}
\definecolor{forestgreen}{rgb}{0.13, 0.55, 0.13}
\definecolor{orange}{rgb}{1,0.49,0}
\date{}

\graphicspath{{figures/}}

\begin{document}

\title{On Structural Rank and Resilience\\ of Sparsity Patterns}
\author{Mohamed Ali Belabbas\thanks{M.-A. Belabbas is with the Department of Electrical and Computer Engineering and the Coordinated Science Laboratory, University of Illinois, Urbana-Champaign. Email:  \texttt{belabbas@illinois.edu}.},\,\,\,\, Xudong Chen\thanks{X. Chen is with the Department of Electrical, Computer, and Energy Engineering, University of Colorado Boulder. Email: \texttt{xudong.chen@colorado.edu}.  },\,\,\,\, Daniel Zelazo\thanks{D. Zelazo is with the Faculty of Aerospace Engineering, Technion-Israel Institute of Technology, Haifa, Israel. Email: \texttt{dzelazo@technion.ac.il}.
 }
 }

\maketitle
\begin{abstract}
\blfootnote{The first two authors contributed equally to the manuscript in all categories.}
A sparsity pattern in $\R^{n \times m}$, for $m\geq n$, is a vector subspace of matrices admitting a basis consisting of canonical basis vectors in $\R^{n \times m}$. We represent a sparsity pattern by a matrix with $0/\star$-entries, where $\star$-entries are  arbitrary real numbers and $0$-entries are equal to $0$.  We say that a sparsity pattern has full structural rank if the maximal rank of matrices contained in it is $n$. 
In this paper, we investigate the degree of resilience of patterns with full structural rank: We address questions such as how many $\star$-entries can be removed without decreasing the structural rank and, reciprocally, how many $\star$-entries one needs to add so as to increase the said degree of resilience to reach a target. Our approach goes by translating these questions into max-flow problems on appropriately defined bipartite graphs. Based on these translations, we provide algorithms that solve the problems in polynomial time. 
\end{abstract}

\section{Introduction}

The development of network-enabled systems~\cite{SavaglioGanzhaPaprzycki2020,  SanislavZeadallyMois2017, 10.1007/978-3-030-52840-9_5, Zhao2019} is creating new opportunities for integrating theretofore disconnected systems. These systems, however, come with new challenges associated to their secure and resilient operation in the face of network-level faults, or even malicious actors intentionally aiming to disrupt their functionality. 
An inherent challenge in problems related to the resilience of these systems against faults or attacks stems from their combinatorial nature, which is induced by the network interconnections.  Indeed, the removal or addition of a communication link in a network is a binary operation, and does not fit well within the framework of, say, robust control theory.  
One closely related problem is the study of structural properties of dynamical systems. The structure is often described by graphs, where the sparsity pattern of the system parameters indicate the presence or absence of edges in an associated graph. 
The so-called \emph{structural system theory} aims to determine whether controllable or stable dynamics can be sustained by a given system structure, which is described via a sparsity patterns for the system matrices~\cite{Lin1974,dion2003generic,BELABBAS2013981,chen2020sparse,gharesifard2021structural, kirkoryan2014decentralized}. 

In this paper, we address a novel structural resilience problem, which is motivated by the recent work~\cite{Sharf2019b} about passivation of networked systems, which will be elaborated on in Subsection~\ref{subsec.motivation}. 
We describe below the mathematical problems addressed in this paper. 
A sparsity pattern is a vector space of matrices admitting a basis comprised only of canonical basis vectors. 
We represent them  as matrices with $0/\star$-entries, where the $\star$ denote arbitrary real entries. 
The starting point of our analysis is to determine  whether a given  sparsity pattern contains an open set of matrices of {\em full rank}. Necessary and sufficient conditions for this requirement to hold are in fact well known and can easily be described using a graph machinery (see Lemma~\ref{lem:frpm}).  
The core problems we address in this paper go beyond that. We consider the {\em resilience} of the full-rank property of these sparsity patterns---here, resilience refers to the property of a sparsity pattern being full-rank after the removal of $\star$-entries (which can be viewed as attacks on communication links). The list of specific problems will be presented in Subsection~\ref{sec:problemformulation}.  
Although the motivation for investigating the above mentioned problems are from passivation of networked systems, they readily apply to other areas, such as the resilience of the structural stability of linear systems~\cite{belabbas_algorithmsparse_2013}. 

Outside of the control theory literature,  problems seeking to understand the structural rank of sparsity patterns have also been addressed in the mathematical literature. In particular, we mention the {\em minimum rank} problem, which aims to determine the minimum rank of real symmetric matrices in a sparsity pattern (where the $\star$-entries {\em have to} be nonzero); see~\cite{Fallat2007, Hogben2010} and the references therein for a comprehensive survey on this subject.  A typical approach to the minimal rank problem involves analyzing a corresponding \emph{inverse problem}, which is trying to identify a graph structure from the spectrum of a matrix~\cite{Gutkin_2001}. 
Some other relevant work include the rank reduction of the adjacency matrix of directed graphs by vertex and/or edge deletions~\cite{Meesum2016b}.

\subsection{Problem Motivation}\label{subsec.motivation}

To illustrate the importance of the sparsity patterns for network systems and their influence on network robustness and resilience, we will look at a general network systems architecture.
The present subsection is thus meant to provide a system theoretic motivation for the problems mentioned above, but the remainder of the paper does not rely on the notions introduced here. 

Consider an ensemble of $n$ agents and $m$ controllers that may exchange state information over a network represented by a matrix $M\in \mathbb{R}^{n\times m}$.  
For ease of exposition, we let the entries of $M$ be either $0$ or $1$.  In this sense, when $M_{ij} = 1$, it means that controller $j$ has access to state information from agent $i$.  The matrix $M$ can therefore also be associated with a graph ${G}=({V},{E})$ with $|{V}| = n+m$ nodes, and edge-set ${E}$ describing the sparsity pattern of~$M$.

For this setup, we assume the agents and the controllers are associated with the dynamical systems $\Sigma_i: u_i \mapsto y_i$ for $i=1,\ldots,n$ and $\Pi_j: \zeta_j \mapsto \mu_j$, for $j=1,\ldots,m$.
Here we assume the agent dynamics and controllers are SISO systems (i.e., $u_i,y_i,\zeta_j,\mu_j\in \mathbb{R}$). The loop is closed by taking $\zeta(t) = M^\top y(t)$ and $u(t) = -M\mu(t)$. This interconnection structure is motivated by the association of each controller with a set of agents. Thus, controller $j$ receives a linear combination of the outputs of adjacent agents (the adjacency relation is encoded in the $j$th column of $M$), and distributes its control output back to the same set of agents. We denote such systems by the triplet $(\Sigma,\Pi,M)$, shown in Fig.~\ref{fig.network}.  Note that if the matrices $M$ are taken to be the \emph{incidence matrix} of a graph $G$, then the system $(\Sigma,\Pi,M)$ describes the well-known diffusively coupled networks \cite{Dorfler2014, Hale1997, Arcak2007}.

\begin{figure}
    \centering
     \subfloat[Block-diagram of the network system $(\Sigma, \Pi, M)$.]{ {\scalebox{.28}{\includegraphics{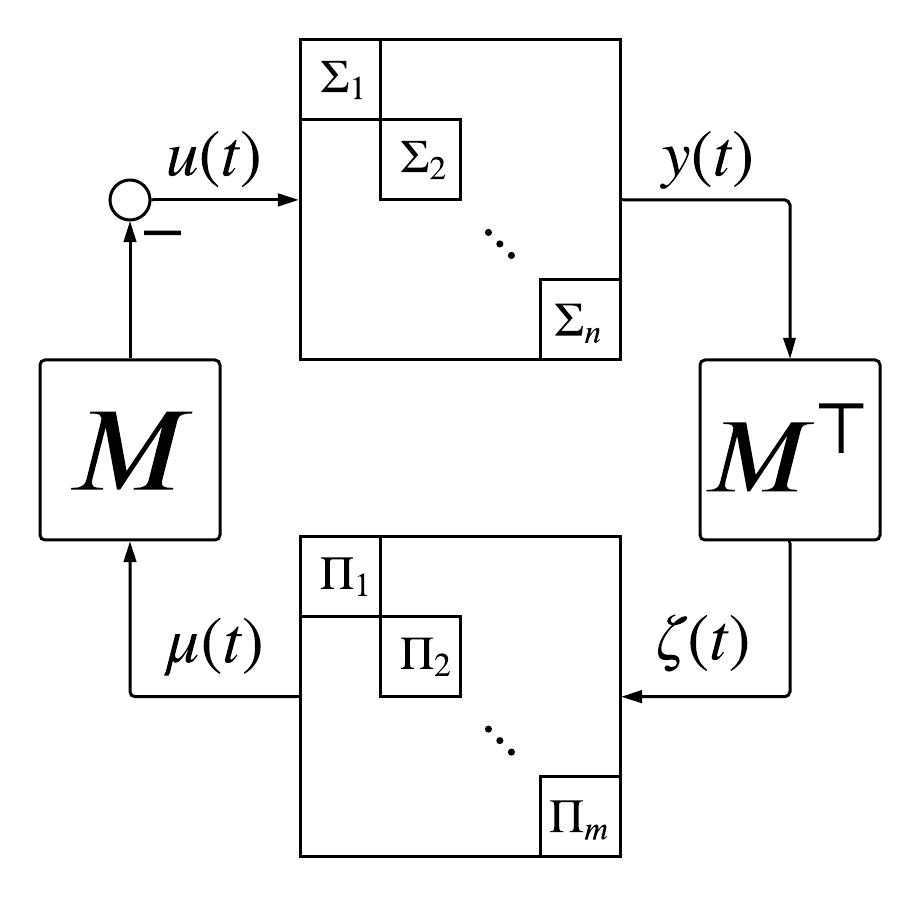}}\label{fig.network} } }
  \hspace{.5cm} \subfloat[Passivation of the system $\Sigma$ over the network interconnection $M$.] {\scalebox{.31}{\includegraphics{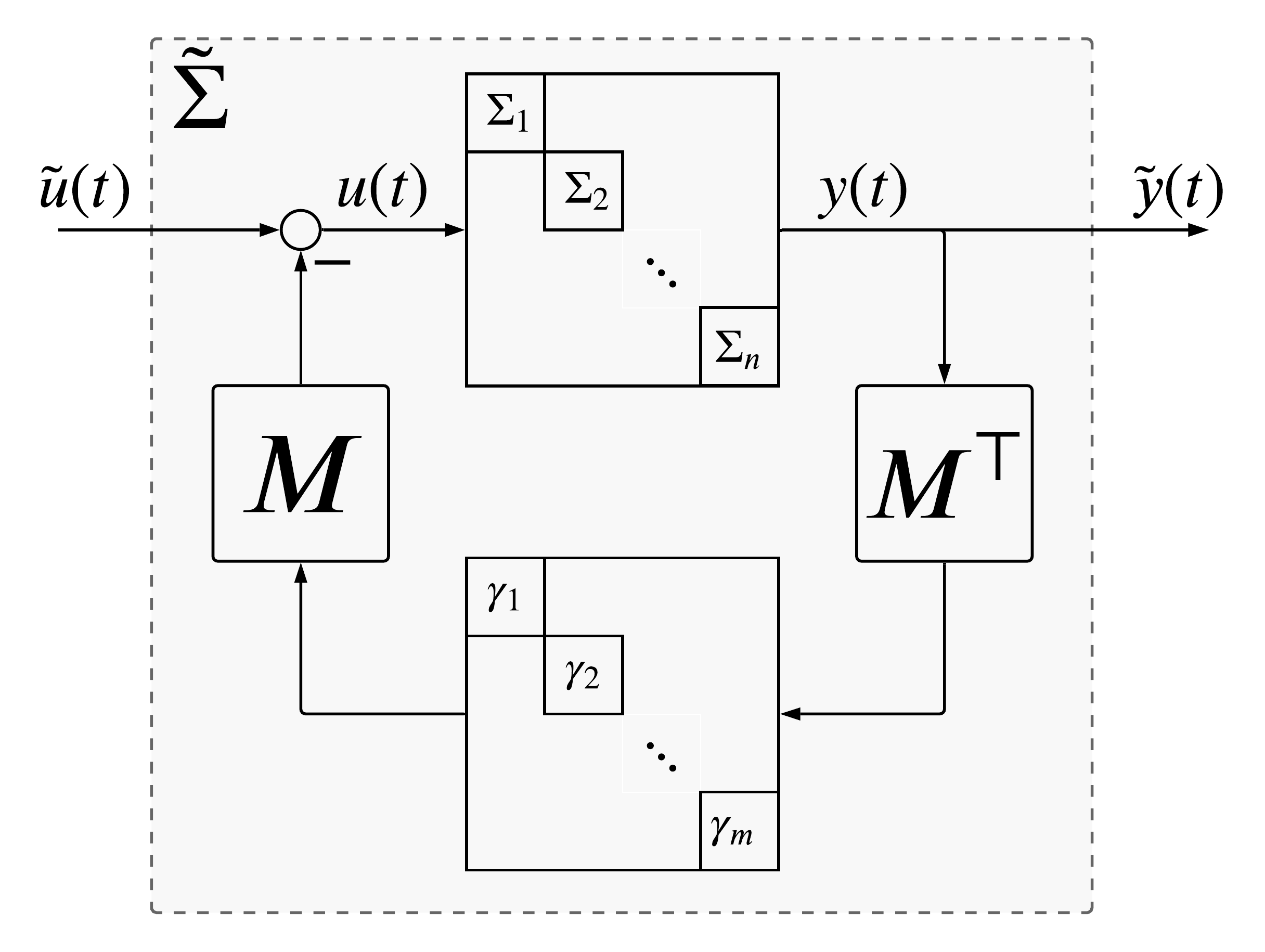}}\label{fig.passify}}
    \caption{A general network system and a network passivation approach.}
    \label{fig.structure}
\end{figure}

The stability of the interconnection in Fig.~\ref{fig.network} can be guaranteed by the (output-strict) passivity of the systems $\Sigma_i$ and passivity of the controllers $\Pi_j$. 

In many applications, however, it may not be possible to guarantee the passivity of the agents $\Sigma_i$. This corresponds to some or all of the agents possessing a negative passivity index; see~\cite{Zhu2014, Sharf2020a} for more details on this notion.  Nevertheless, it is still desirable to be able to interconnect these so-called passive-short systems with each other to achieve group coordination tasks. 
In this direction, there have been recent works that aim to passify these agents over the network itself \cite{Atman2018, Kojima2019, Sharf2019b}.   This architecture can be seen in Fig.~\ref{fig.passify}, where the gains $\gamma_i$ are chosen to ensure the system from input $\tilde{u}$ to output $\tilde{y}$ is passive.  
If this can be achieved, then it can be shown that the network interconnection $(\tilde{\Sigma}, \Pi, M)$ is stable, where $\tilde{\Sigma}$ maps $\tilde{u}$ to output $\tilde{y}$~\cite{Sharf2019b}. The conditions for which this is possible were explored in~\cite{Sharf2019b}.  The main result can be extended to the general network structure $M$, so we state it below without a proof:

\begin{lemma}\label{cor.weights}
Let $R=\diag(\rho_1,\ldots,\rho_n)$ be a diagonal matrix containing the passivity index $\rho_i$ of each agent $\Sigma_i$, and assume that $\rho_i<0$ for at least one agent. If $R + M\diag(\gamma)M^\top$ is positive-semi definite, then $\tilde{\Sigma}$, mapping $\tilde{u}(t)$ to $\tilde{y}(t)$ as in Fig. \ref{fig.passify}, is passive with respect to any steady-state input-output pair. Moreover, if $R + M\diag(\gamma)M^\top$ is positive-definite, then $\tilde{\Sigma}$ is output-strict passive. 
Furthermore, there exist scalars $\gamma_i$, for $i=1,\ldots,m$, such that $R+M\diag(\gamma)M^\top > 0$ if and only if $x^\top Rx>0$ for any $x \in \ker(M^\top)$.
\end{lemma}
This result shows that for a given network matrix $M$, it may not even be possible to guarantee a network passivation scheme that ensures output-strict passivity of $\tilde{\Sigma}$.  At the same time, it hints that for a given set  of passivity indices $\rho_i$, a change to the network matrix $M$ may allow for output-strict passivation.  This result also shows that for a \emph{full-rank} matrix $MM^\top$, it will always be possible find a single gain $\gamma$ such that $R+\gamma MM^\top>0$.

With this setup, we can now motivate the study of the \emph{structural rank} of the interconnection matrix $M$.  For a matrix $M$ with a given sparsity pattern, how many of its entries can be removed, corresponding to compromising the network connection between an agent and controller, before the matrix loses rank.  In the case where the network is being used to also passify the agents, this loss of rank may lead to the loss of passivity of $\tilde{\Sigma}$, thereby destroying the convergence guarantees of the network system $(\tilde{\Sigma}, \Pi,M)$. To illustrate this, we present a brief example.
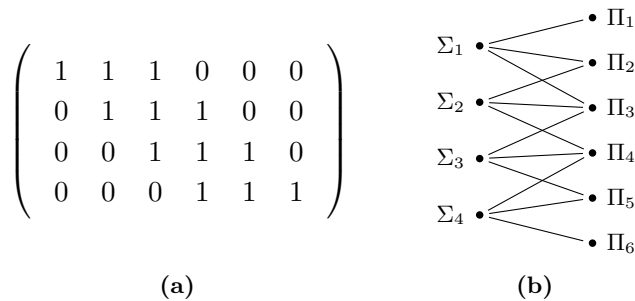
\begin{figure}[!b]
    \centering
    \subfloat[]{
   
 \begin{tikzpicture}[scale = 0.25]
 \matrix [{matrix of math nodes}, column sep=5pt, row sep=1pt,
     left delimiter=(,right delimiter=),ampersand replacement=\&] (m)
 {
  1 \& 1 \& 1 \& 0 \& 0 \& 0   \\
  0 \& 1 \& 1  \& 1 \& 0 \& 0\\
  0 \& 0 \& 1 \& 1 \& 1 \& 0\\
  0 \& 0 \& 0 \& 1 \& 1 \& 1 \\ 
 };
 \node[fit=(m-1-1) (m-1-3)] {};
 \node[fit=(m-2-4) (m-2-4)] {};
 \node[fit=(m-3-5) (m-3-5)] {};
 \node at (0,-6.45){};
 \end{tikzpicture} 
 }\qquad
     \subfloat[]{

 \includegraphics{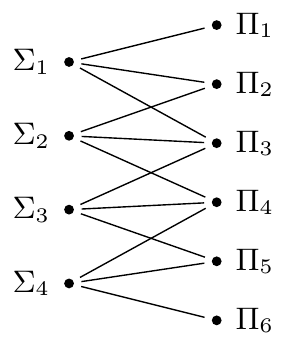}
} 
    \caption{The network matrix $M$ in (a) and its graph representation in (b). Each column of $M$ represents a controller $\Pi_e$ while each row  corresponds to a system $\Sigma_i$.}
    \label{fig:Mexample}
\end{figure}

\begin{example}\normalfont
We consider an ensemble of $n=4$ identical, but unstable plants, with dynamics of each agent described by the SISO transfer function $\Sigma_i(s)=(s+0.5)/(s-1)$ for $i\in \{1,\ldots,4\}$. It can be verified that the agents are output passive-short, with $\rho = -2$.\footnote{The passivity index can be computed, for example, in MATLAB using the command \texttt{getPassiveIndex}.}  The agents are to be controlled according to the architecture in Fig. \ref{fig.network} with the network matrix $M$, illustrated in Fig. \ref{fig:Mexample} below. Since $M$ is full rank, according to Lemma \ref{cor.weights}, the ensemble can be passified (and stabilized) over the network using the architecture in Fig. \ref{fig.passify} and with gain $\gamma > 3.4142$ (found using, for example, semi-defnite programming).  

Consider now a scenario where an attacker successfully disables controllers $\Pi_1$ and $\Pi_6$ (corresponding to nulling columns $1$ and $6$ in $M$).  Even whith such an attack, the matrix $M$ maintains full column rank and can still be passified, now with a gain of $\gamma > 13.7082$.  On the other hand, if in addition the connection between $\Sigma_3$ and $\Pi_3$ is severed (i.e., changing entry $M_{13}$ to $0$), then $M$ loses rank and it is no longer possible too passify the system over the network. Consequently, the architecture of Fig. \ref{fig.network} can not be used to control the ensemble and the attacker was successful in disabling the system. \qed
\end{example}

\subsection{Problem formulation and contributions}\label{sec:problemformulation}

In this subsection, we formulate the core problems addressed in this paper.  We start by introducing the notions of sparsity patterns and their rank.

A \emph{sparsity pattern} $\cS(n,m)$ in $\R^{n \times m}$ (or simply $\cS$ if $(n,m)$ is clear from the context) is a vector subspace that admits a basis consisting only of matrices $E_{ij}$'s, i.e., matrices with $1$ on the $ij$th entry and $0$ elsewhere. Such a vector space is thus fully determined by the pairs $(i,j)$, which indicate the entries of matrices in $\cS$ that are not always zero. We denote by $E(\cS)$ the collection of all such pairs, hence $\dim \cS = |E(\cS)|$. We refer to the entries of $\cS(n,m)$ indexed by $E(\cS)$ as $\star$-entries, and the other entries as $0$-entries.

\begin{definition}[Rank of sparsity pattern]
The rank of a sparsity pattern $\cS$, denoted by $\rk \cS$, is the maximal value of the ranks of matrices in $\cS$:
$$
\rk \cS := \max_{A \in \cS}(A).
$$
\end{definition}

It should be clear that $\rk \cS(n,m) \le \min\{n, m\}$. Returning to the example of Section \ref{subsec.motivation}, we are interested in finding sparse matrices $M$ such that $MM^\top $ is full rank (i.e., rank $n$).  Thus, we assume in the sequel that $m \geq n$. 

The set of sparsity patterns of the same parameters $(n,m)$ admits a natural partial order:
\begin{definition}[Partial order on sparsity patterns]
Given patterns $\cS(n,m)$ and $\cS'(n,m)$, we write $\cS' \succeq \cS$ if $E(\cS') \supseteq E(\cS)$ and $\cS' \succ \cS$ if $E(\cS') \supsetneq E(\cS)$. 
\end{definition}

We now precisely  define the notions of resilience studied in this paper:

\setcounter{definition}{3}
\begin{subdefinition}[Resilience]\label{def:res}
Given positive integers $n$ and $m$ with $m\ge n$, a sparsity pattern $\cS(n,m)$ of rank $n$ is \emph{exactly $k$-resilient}, for $0 \leq k\leq |E(\cS)|$, if the following hold: 
\begin{enumerate}
    \item All patterns $\cS'\prec \cS$ with $|E(\cS')| \geq |E(\cS)|-k $ are of rank~$n$; 
    \item There exists an $\cS'\prec \cS$ with $|E(\cS')| = |E(\cS)|-k-1 $ whose rank is less than $n$.
\end{enumerate}
We denote by $\rs \cS$ the degree of resilience of $\cS$.
\end{subdefinition}
Note that by the above definition, a sparsity pattern $\cS$ is exactly $0$-resilient if its rank is~$n$ and, moreover, all patterns $\cS' \prec \cS$ have ranks strictly lower than $n$.

When expressing $\cS$ as a direct sum  
of sparsity patterns, it is clear that if {\em any} of the summand is of rank $n$, then so is $\cS$. Following this fact, we introduce the following definition:

\begin{subdefinition}[Strong resilience]\label{def:strongres}
Given positive integers $n$ and $m$ with $m\ge n$, a sparsity pattern $\cS(n,m)$ of rank $n$ is {\em exactly  strongly $k$-resilient}, for $k\geq 0$, if it contains a {\em direct sum}\footnote{The {\em direct} sum  $\cS'\oplus \cS''$, for $\cS', \cS''\subseteq \cS$, is well defined only if $\cS' \cap \cS'' =\{0\}$.} of $(k+1)$, but not $(k + 2)$, sparsity patterns each of which is $0$-resilient.
We denote by $\srs \cS$ the degree of strong resilience of $\cS$. 
\end{subdefinition}

On occasions, we will deal with sparsity patterns $\cS(n,m)$ that are not full rank, i.e., $\rk \cS(n,m)<n$. By convention, we set
\begin{equation}\label{eq:negativesrs} 
\mrs \cS(n,m):=-1 \mbox{ if } \rk
\cS(n,m)<n.
\end{equation}

Throughout this paper, we shall always consider the {\em exact} degree of (strong) resilience of a sparsity pattern. Thus, for convenience, we will omit ``exact'' in the sequel if there is no confusion.   

By the arguments outlined  before Def.~\ref{def:strongres}, if a sparsity pattern is strongly $k$-resilient, then it is at least  $k$-resilient. 
However, the converse is not true: there exist $k$-resilient patterns which cannot be expressed as a direct sum of $(k+1)$ patterns which are $0$-resilient 
(an example is given in Fig.~\ref{fig:weakvsstrong} below).  
Nevertheless, we will show in Section~\ref{ssec:resandstrongres} that the gap between the two notions does not have any impact on the minimal dimensions of patterns meeting either definition. Specifically,   
if $d_k$ is the {\em minimal dimension} of a $k$-resilient pattern,
$$ 
d_k := \min_{\cS\,: \,\rs \cS =k} |E(\cS)|,
$$ 
then there exists a pattern of dimension $d_k$ which is {\em strongly $k$-resilient}. 

Standing from the perspective of a system designer, we pose the following questions: 
\begin{enumerate}
\item[P1:] Given a sparsity pattern $\cS$, what is its degree of (strong) resilience? 
\item[P2:] Given a sparsity pattern $\cS$, what is the least number of $\star$-entries one should add to obtain a degree of (strong-)resilience $k^*$? This problem can be expressed as follows: 
    \begin{equation}
        \min |E(\cS^*)| \mbox{ s.t. }   \cS^* \succeq \cS \mbox{ with  } \mrs \cS^* =k^*.
    \end{equation}
\item[P3:] Given a sparsity pattern $\cS$, what is the largest degree of (strong-)resilience we can achieve by adding $p$ $\star$-entries? This problem can be expressed as follows:
    \begin{equation}
        \max \mrs \cS^*  \mbox{ s.t. }  \cS^* \succeq \cS \mbox{ with  } |E(\cS^*)|=|E(\cS)|+p.
    \end{equation}
\end{enumerate}

We provide solutions to the three problems stated above, as well as relevant polynomial-time algorithms,  with a focus on strong resilience. The solutions are formulated in Theorem~\ref{th:maxflowgood}, Theorem~\ref{thm:solutiontop2}, and Theorem~\ref{thm:solutiontop3}, respectively.

The first step in our analysis is to assign a bipartite graph to a sparsity pattern, and to relate the full-rank property, and its resilience,  to the existence of matchings in this graph. This is done in Sec.~\ref{sec:patternstographs}. We then proceed toward the first  result,  Theorem~\ref{thm:exactkmatching}, in which relying on a result of K\"onig~\cite{konig1916graphen} to characterize the bipartite graphs corresponding to strongly $k$-resilient patterns. This is done in Sec.~\ref{ssec:resandstrongres}.
Relying on Theorem~\ref{thm:exactkmatching}, we then translate the three problems formulated above into problems about \emph{max-flows} over graphs. In more details,  we first create several variations on the bipartite graphs associated with a sparsity pattern, by adding source and target nodes, turning undirected edges into directed ones, and appropriately assigning edge- and node-capacities. We then introduce several max-flow problems defined on these modified bipartite graphs and, moreover, prove that integral solutions to these max-flows problems provide  solutions to the original problems P1-P3. Finally, we demonstrate that these max-flow problems can be solved using standard algorithms in polynomial time.

\section{Bipartite graphs, matchings, and resilience}\label{sec:patternstographs}

\subsection{Background on graph theory and flows}\label{ssec:background}
We introduce the necessary background and  notations about graph theory and related flow problems. We will be concerned in this paper with bipartite graphs, i.e., graphs which admit a partition of their node set into two disjoint components with the property that nodes in the same components share no edge.

Denote by $G(n,m)=(V_\alpha \cup V_\beta, E)$ an undirected bipartite graph on $(n + m)$ nodes: by convention, there are $n$ \emph{left-nodes} denoted by $\alpha_1,\ldots, \alpha_n$ and $m$ \emph{right-nodes} denoted by $\beta_1,\ldots, \beta_m$. On occasions, we will write $G$ by omitting the arguments $(n,m)$ if it is clear from the context.   
Each edge of $G(n,m)$ connects a left-node with a right-node. An edge in $G(n,m)$ is thus denoted by $(\alpha_i,\beta_j)$. We say that $G'=(V',E')$ is a subgraph of $G$ if $V' \subseteq V_\alpha \cup V_\beta$ and $E' \subseteq E$. Given a node $\alpha$ in $G$, we denote by $\deg(\alpha, G')$ the degree of $\alpha$ relative to $G'$, defined as the number of edges in $E'$ incident to $\alpha$ (equivalently, the number of neighbors of $\alpha$ in $G'$). We will also consider below {\em directed} bipartite graphs; we denote  the {\em directed} edge from $\alpha_i$ to $\beta_j$ by $\alpha_i\beta_j$.

Recall that a {\em matching} in the graph $G(n,m)$ is a set of edges so that no two distinct edges are incident to the same node. 
For $n = m$,  a {\em perfect matching} $P$ in $G(n,n)$ is a set of $n$ edges such that each node of $G(n,n)$ is incident to exactly one of these $n$ edges. 
For the general case $m \neq n$, we introduce the following definition: 

\begin{definition}[Left-perfect matchings]\label{def:lrperfectmatchings}
A {\em left-perfect matching} in a bipartite graph $G(n,m)=(V_\alpha\cup V_\beta,E)$, with $m\ge n$, 
is a set of $n$ edges in $E$ so that no two distinct edges are incident to the same node. 
\end{definition}

Equivalently, $G(n,m)$, for $n\leq m$, admits a left-perfect matching if there exist $n$ distinct right nodes $\beta_{i_1},\ldots, \beta_{i_n}$ such that the subgraph $G'(n,n)$ induced by the left-nodes $V_\alpha$ and $\{\beta_{i_1},\ldots, \beta_{i_n}\}$ has a perfect matching. 
We say that two matchings $P_1$ and $P_2$ of $G$ are {\em disjoint} if $P_1 \cap P_2 = \varnothing$.

Let $\vec G = (V, \vec E)$ be an arbitrary directed graph, with two special nodes $s,t \in V$, termed the {\em source} and {\em target} nodes, respectively. The source node has no incoming edges and the target has no outgoing edges. 
A {\em capacity} on $\vec G$ is a function $c: \vec E\to \R_{\ge 0}$.  
Given the capacity, a {\em flow}  on $\vec G$ is a function $f: \vec E\to \R_{\ge 0}$ such that

\begin{enumerate}
\item $f(e) \leq c(e)$ for all $e\in \vec E$;
\item the following balance condition is satisfied at all nodes $v \in V-\{s,t\}$:
\begin{equation}\label{eq:flowcond}
\sum_{u:uv\in \vec E} f(uv) = \sum_{w:vw\in \vec E}f(vw). 
\end{equation}
\end{enumerate}
The {\em value} of the flow $f$ is defined as 
\begin{equation}\label{eq:defvalueofflow}
|f|:= \sum_{v:sv\in \vec E} f(sv) = \sum_{v:vt\in \vec E}f(vt).
\end{equation}
We denote by 
$F_c$ the set of all flows on $\vec G$ with  capacity  function $c$. 
The so-called {\em max-flow problem} \cite{Rockafellar1998} is the optimization problem formulated as follows: 
$$\max_{f \in F_c} |f|.$$ 
It is well known that finding a solution $f^*$ to the above optimization problem can be done in polynomial time using, e.g., the Ford-Fulkerson algorithm~\cite{ford2015flows}. 
Note that a solution to the max-flow problem is not necessarily integer-valued, i.e., there may exist edges $e$ such that $f^*(e)$ are not integers, even if $c$ is integer-valued. 
However, if $c$ is integer-valued (which will be the case in this paper),  then the output of the Ford-Fulkerson algorithm initialized at an integer-valued flow is  integer-valued as well, and thus provides an {\em integer-valued} maximum flow \cite{Rockafellar1998}. This statement is referred to as the {\em integrality theorem}.

A fundamental result in the study of max-flow problems is the \emph{max-flow min-cut theorem}, which we briefly describe here. 
To this end, we recall the definition of a {\em cut} in the digraph $\vec G = (V, \vec E)$ with the capacity function $c$: 
An $s$-$t$ cut $(S,T)$ in $\vec G$ is a partition of the node set of $\vec G$ into two disjoint sets $S \ni s$ and $T \ni t$. We denote the set of all $s$-$t$ cuts in $\vec G$ as $\cal C$. For a given cut $(S,T) \in \cal C$, we let 
$$\vec E_{(S,T)}:= \{v_iv_j\in \vec E \mid v_i\in S, v_j \in T\}.$$
Then, the {\em capacity} of the cut $(S,T)$ is defined as 
$$
c(S, T):= \sum_{e\in \vec E_{(S,T)}} c(e).
$$
The min-cut problem is then formulated as follows
$$
\min_{(S,T)\in {\cal C}} c(S,T).
$$
The max-flow min-cut Theorem~\cite{brualdi2010introductory} says the following: 

\begin{lemma}\label{lem:maxflowmincut}
 Given a digraph $\vec G$ with source $s$ and target $t$ and capacity function $c$,  let $F_c$ be the set of corresponding flow maps on $\vec G$ and $\cal C$ the set of $s$-$t$ cuts in $\vec G$. Then,  
 $$
 \max_{f \in F_c} |f| = \min_{(S,T)\in {\cal C}} c(S,T).
 $$
\end{lemma}

\subsection{Graph theoretic view on (strong) resilience}

To proceed, we establish some standard connections between graph theoretic concepts and the pattern properties introduced here. 
First, to a given  sparsity pattern $\cS(n,m)$, we can assign the bipartite graph $G(n,m)=(V_\alpha \cup V_\beta, E)$ on $(n + m)$ nodes  with edge set $E$ given by the rule: 
the $ij$th entry of $\cS$ is  a $\star$ if and only if  $(\alpha_i, \beta_j)$ is an edge in $E$. See Fig.~\ref{fig:matrixtograph} for an illustration. 

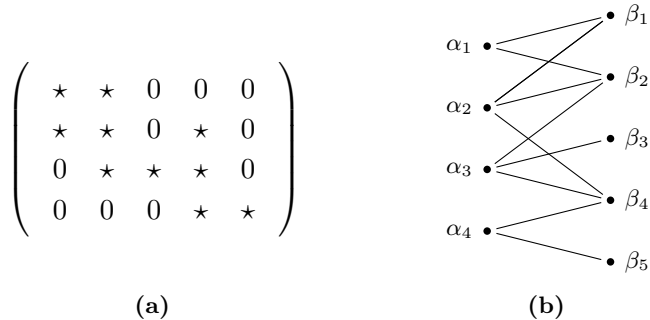
\begin{figure}
    \centering
    \subfloat[]{
 \begin{tikzpicture}[scale = 0.82]
 \matrix [{matrix of math nodes}, column sep=5pt, row sep=1pt,
     left delimiter=(,right delimiter=),ampersand replacement=\&] (m)
 {
  \star \& \star \& 0 \& 0 \& 0 \\
  \star \& \star \& 0 \& \star \& 0 \\
  0 \& \star \& \star \& \star \& 0 \\
  0 \& 0 \& 0 \& \star \& \star \\ 
 };
 \node[fit=(m-1-1) (m-1-3)] {};
 \node[fit=(m-2-4) (m-2-4)] {};
 \node[fit=(m-3-5) (m-3-5)] {};
 \node at (0,-2) {};
 \end{tikzpicture} 
 } 
 \qquad \qquad
     \subfloat[]{
 \includegraphics{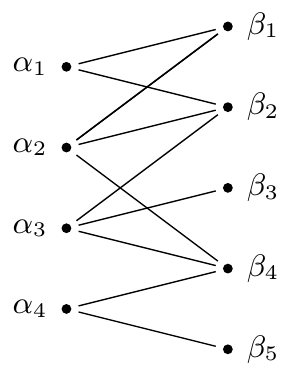}
} 
    \caption{We illustrate the correspondence between a sparsity pattern in (a) and a bipartite graph in (b). The left- and right-nodes are labeled as $\alpha_i$ and $\beta_j$, respectively. 
    A star in the $ij$th entry corresponds to the edge $(\alpha_i,\beta_j)$ in the bipartite graph.}
    \label{fig:matrixtograph}
\end{figure}

Since this representation of sparsity patterns as bipartite graphs is one-to-one,  we also write $\mrs G$ to refer to the degree of (strong) resilience of the corresponding pattern $\cS$.
We now relate $\mrs G$ to perfect matchings of $G$.  
The following result is standard, and we include a proof in the Appendix for completeness.

\begin{lemma}\label{lem:frpm}
A sparsity pattern $\cS(n,m)$  is of rank $n$ if and only if its associated bipartite graph $G(n,m)$ admits a left-perfect matching.  
\end{lemma}

As an immediate consequence of the above lemma, we can characterize $k$-resilient bipartite graphs  as follows:

\begin{lemma}
 A bipartite graph $G = (V_\alpha\cup V_\beta, E)$ is $k$-resilient if and only if the following hold: 
 \begin{enumerate}
 \item For {\em any} subset $E' \subset E$ with $|E'|=k$, $G'=(V_\alpha\cup V_\beta, E-E')$ contains a left-perfect matching; \item There exists a subset $E'$ with $|E'| = k+1$ such that $G' = (V_\alpha\cup V_\beta, E - E')$ does not contain a left-perfect matching. 
 \end{enumerate}
\end{lemma}

We can also characterize strongly $k$-resilient bipartite graphs using perfect matchings:   

\begin{lemma}\label{lem:reskpm}
A bipartite graph $G$ is strongly $k$-resilient  if and only if it has exactly $(k+1)$ {\em disjoint} left-perfect matchings.
\end{lemma}

\begin{proof}
We first show that if $G$ has exactly $(k+1)$ disjoint left-perfect matchings, it is strongly $k$-resilient. Denote by $P_1,\ldots,P_{k+1}$ the disjoint left-perfect matchings in $G$. By Lemma~\ref{lem:frpm}, the graph induced by each left-perfect matching in $G$ corresponds to a $0$-resilient sub-pattern of $\cS$. Furthermore, since the $(k+1)$ left-perfect matchings are disjoint, the sparsity pattern corresponding to their union is the direct sum of the sub-patterns corresponding to the $P_i$. 
It then follows from Definition~\ref{def:strongres} that $G$ is strongly $k$-resilient.  

We now show that if $G$ is strongly $k$-resilient, then it has exactly $(k+1)$ disjoint left-perfect matchings. First, note that $G$ cannot have more than $(k+1)$ disjoint left-perfect matchings because otherwise, by the above argument, $G$ is at least strongly $(k+1)$-resilient. 
It remains to show that $G$ has at least $(k+1)$ disjoint left-perfect matchings. 
By definition of strong resilience, $\cS$ contains $(k+1)$ subpatterns $\cS_1, \ldots, \cS_{k+1}$  that are $0$-resilient and $\cS_i \cap \cS_j = \{0\}$ for $i\neq j$. 
Owing to the correspondence between sparse patterns and bipartite graphs, to each subpattern corresponds a subgraph of $G$. Denote these subgraphs by $G_1,\ldots, G_{k+1}$. Since each pattern is $0$-resilient, by Lemma~\ref{lem:frpm}, each $G_i$ contains at least one left-perfect matching $P_i$.  
Since $\cS_i \cap \cS_j=\{0\}$ for $i \neq j$, it follows that $G_i$ and $G_j$ are edge-wise disjoint and, hence, $P_i$ and $P_j$ are disjoint as well. We have thus shown that $G$ has {at least} $(k+1)$ disjoint left-perfect matchings, which concludes the proof.
\end{proof}

\section{Main Results}\label{sec:mainresults}

\subsection{On $k$- and strong $k$-resilience}\label{ssec:resandstrongres}

From Lemma~\ref{lem:reskpm}, we know that a strongly $k$-resilient pattern is associated to a bipartite graph that contains exactly $k$ disjoint left-perfect matchings.  To better understand strong resilience, we characterize graphs that can be obtained as unions of disjoint left-perfect matchings:

\begin{theorem}\label{thm:exactkmatching}
A bipartite graph $G(n,m)$, for $m\ge n$, is a union of $k$, for $1\leq k \leq m$, disjoint left-perfect matchings if and only if the following hold:
\begin{enumerate}
    \item The degree of each left-node is exactly $k$;
    \item The degree of each right-node is less than or equal to~$k$.
\end{enumerate}
\end{theorem}

Note that the degree of each right node of $G(n,m)$ is at most $n$, so for $k\ge n $, item (2) of Theorem~\ref{thm:exactkmatching} holds trivially. 
It is not too hard to see the bipartite graphs characterized by Theorem~\ref{thm:exactkmatching} exist for every $k = 1,\ldots, m$. 

\begin{proof}
We first establish the necessity of the two items. 
The necessity of item (1) is obvious. For item (2), assume, to the contrary, that there is at least one node in $V_\beta$, say $v_{\beta_j}$, with degree larger than $k$. 
Since each node of $V_\beta$ is incident to at {\em most one} edge in a left-perfect matching, after removing the $k$ disjoint perfect matchings of $G(n,m)$,  $v_{\beta_j}$ will have degree strictly larger than $0$ and thus $G(n,m)$ is not a union of $k$ left-perfect matchings.

We next establish the sufficiency of the two items. 
The proof relies on the use of K{\"o}nig's Line Coloring Theorem~\cite[Theorem 1.4.18]{lovasz2009matching}, which can be equivalently stated as follows: Let $G=(V_\alpha\cup V_\beta,E)$ be an arbitrary bipartite graph, and $\Delta(G)$ be the maximal degree of $G$, i.e., $\Delta(G):= \max_{v\in V_\alpha \cup V_\beta}\deg(v)$. Further, let $\chi(G)$ be the minimal number $\ell$ of {\em disjoint} matchings $P_1,\ldots,P_\ell$ in $G$ such that $E = \cup_{i = 1}^\ell P_i$. Then, $\chi(G) = \Delta(G)$. 
Applying K{\"o}nig's Line Coloring Theorem to $G(n,m)$ as in the theorem statement, we obtain that $E$ is a union of $k$ disjoint matchings $P_1,\ldots,P_k$. In order to show that these matchings are all left-perfect matchings, it suffices to show that they are all of cardinality $n$. Indeed, since $n\leq m$, any matching of cardinality $n$ is necessarily left-perfect. Note that $|E| = \sum_{i = 1}^k |P_i|$ and, by the hypothesis on $G(n,m)$, $|E| = kn$. Finally, since $G$ is bipartite, the cardinality of {\em any} matching in $G$ cannot exceed $n$. We conclude that all matchings $P_1,\ldots, P_k$ have cardinality $n$ and are thus left-perfect matchings. 
\end{proof}

The following result is a corollary of Theorem~\ref{thm:exactkmatching}: 

\begin{corollary}\label{cor:littlebc}
The following two statements hold:
\begin{enumerate}
\item For any given $k = 1,\ldots, m-1$, the minimal number of edges needed for $G(n,m)$, with $m\geq n$, to be  $k$-resilient (or strongly $k$-resilient) is $(k+1)n$. 
\item  Given a pair of positive integers $(n,m)$ with $m \ge n$, the maximal degree of resilience (or strong resilience) of a bipartite graph $G(n,m)$ is $(m-1)$. 
\end{enumerate}
\end{corollary}
\begin{proof}
We first establish the fact that if a bipartite graph $G(n,m)$ is $k$-resilient, then it has at least $(k+1)n$ edges. To see this, recall that by Lemma~\ref{lem:frpm}, $G(n,m)$ is $k$-resilient if, after removing $k$ edges, the remaining graph still admits a left-perfect matching. Hence, the degree of each left node has to be at least $(k+1)$ because otherwise, such node can be disconnected from the others by the removal of $k$ edges incident to it. Since $G(n,m)$ is bipartite, this proves the claim.
Item 1 is then an immediate consequence of the above fact and Theorem~\ref{thm:exactkmatching}. 

We now prove item 2. To consider maximal degree of (strong) resilience,  it suffices to let $G(n,m)$ be the complete bipartite graph (owing to the monotonicity of resilience with respect to adding edges). In this case, we show that the degree of (strong) resilience of $G(n,m)$ is $(m-1)$.  
On one hand, the degree of every left node is $m$. From the fact established at the beginning of the proof, we have that $G(n,m)$ is at most $(m-1)$-resilient. On the other hand, by Theorem~\ref{thm:exactkmatching}, $G(n,m)$ is a union of $m$ disjoint left-perfect matchings. 
Thus, by Lemma~\ref{lem:reskpm}, $G(n,m)$ is strongly $(m-1)$-resilient.  
\end{proof}

The above corollary says that $k-$ and strongly $k$-resilience require the same minimal number of edges, and that the maximal degrees of resilience and of strong resilience one can achieve for a given $(n,m)$ are also the same. Nevertheless, they are distinct notions: strong $k$-resilience is strictly stronger than $k$-resilience.  
We provide an example in Fig.~\ref{fig:weakvsstrong} where a graph that is $1$-resilient but strongly $0$-resilient is depicted.

\begin{figure}[ht!]
    \centering
    \subfloat[]{
    \includegraphics{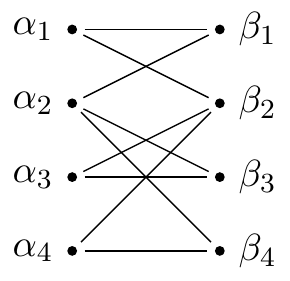}
}\qquad \qquad
\subfloat[]{
\includegraphics{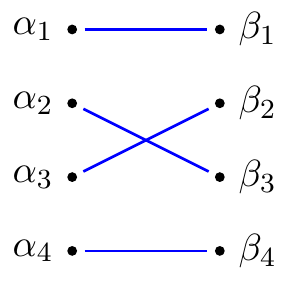}
}\\

\subfloat[]{
  \includegraphics{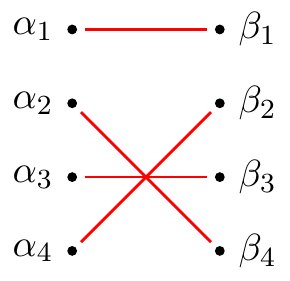}
}\qquad\qquad
\subfloat[]{
 \includegraphics{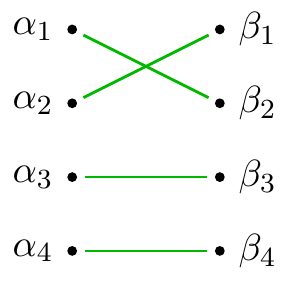}
}
    \caption{The graph in (a) contains three distinct (but not pairwise disjoint) perfect matchings, depicted in (b)-(d). The intersection of these three perfect matchings is the empty set, i.e., there is no common edge to all these matchings. Hence, the graph in (a) is at least $1$-resilient. Moreover, since $ \deg(\alpha_i)=\deg(\beta_i) =2$ for $i = 1,3,4$, the graph in (a) is exactly $1$-resilient. However, since the pairwise intersections of the matchings is not empty,  it is not {\em strongly} $1$-resilient.}
    \label{fig:weakvsstrong}
\end{figure}

In the sequel, we will mostly focus on {\em strong} resilience. The main reason for this is the characterization provided by Theorem~\ref{thm:exactkmatching}, which we can leverage to obtain provable solutions to problems P1-P3. An equivalent characterization for resilience appears harder to obtain.  While the previous example shows that resilience can be strictly weaker than strong resilience, Corollary~\ref{cor:littlebc} shows that the two notions are interchangeable when the number of edges used (which can be viewed as resources deployed by the designer) is to be minimized. 

\subsection{Solution to Problem P1}\label{ssec:maxflow}

In this section, we show how to determine the degree of strong resilience of a bipartite graph $G(n,m)$ for $n\le m$; i.e., we provide a solution to Problem P1. The solution is constructive, in the sense that we also exhibit a set of edges which is a union of disjoint left-perfect matchings, and can be obtained in polynomial-time. 
This is done by translating the problem into a max-flow problem and appealing to Theorem~\ref{thm:exactkmatching}. 

We start with the following definition, which takes a bipartite graph $G$ and a nonnegative integer $\ell$ and produces a directed version of $G$, denoted by $\bar G$, and a capacity function defined on the edge set of $\bar G$:  

\begin{definition}\label{def:alg1}
Let $G(n,m) =(V_\alpha \cup V_\beta,E)$ be a bipartite graph and $\ell\geq 0$ an integer. Define the digraph $\bar G(n,m)=(\bar V, \bar E)$ and the  capacity function $\bar c_\ell: \bar E\to \Z_{\geq 0}$ as follows:
\begin{enumerate}
    \item Add two new nodes to $G$, denoted by $s$ and $t$: $$\bar V := V_\alpha \cup V_\beta \cup \{s, t\}.$$
    \item Create the edge set $\bar E$ as a union of $\bar E_0$ and $\bar E_1$ where  
\begin{equation*}\label{eq:defE0E1}
\begin{aligned} \bar E_0 &:= \{s\alpha_i, \beta_j t \mid \alpha_i \in V_\alpha, \beta_j \in V_\beta\}, \\ 
\bar E_1 &:= \{ \alpha_i\beta_j \mid (\alpha_i,\beta_j) \in E\}.
\end{aligned}
\end{equation*}
\item Define $\bar c_\ell$ as follows:  
\begin{equation}\label{eq:capacity}
\bar c_\ell(e):= 
\begin{cases}
\ell & \mbox{if } e\in \bar E_0, \\
1 & \mbox{if } e\in \bar E_1.
\end{cases} 
\end{equation} 
\end{enumerate}
\end{definition}

The two new nodes $s$ and $t$ added in step 1 are the source and the target of $\bar G$, respectively. The value of $\ell$ will be problem-dependent and specified below. We illustrate the definition in Fig.~\ref{fig:algorithm1}.

\begin{figure}
    \centering
    \subfloat[]{
 \includegraphics{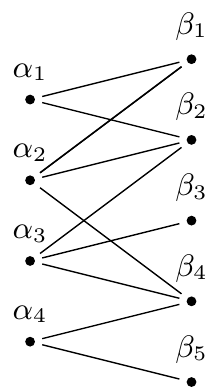}
} \qquad \qquad
\subfloat[]{
 \includegraphics{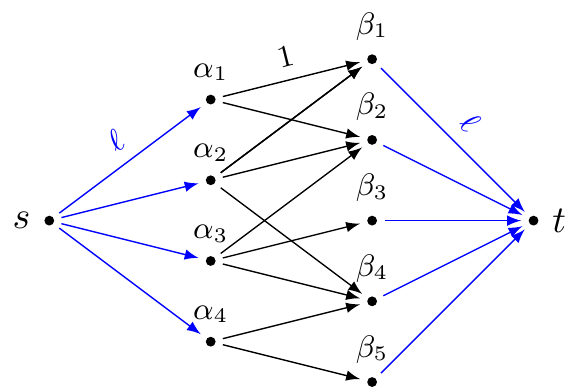}
}
    \caption{Given a bipartite graph $G(4,5)$ in (a), we plot the digraph $\bar G(4,5)$ in (b). The edge set $\bar E$ of $\bar G(4,5)$ is partitioned into two subsets $\bar E_0$ and $\bar E_1$ depicted in blue and black, respectively. The capacity of each blue (resp. black) edge is $\ell$ (resp. $1$).}
    \label{fig:algorithm1}
\end{figure}

Denote by $\fl$ the set of integer-valued flow maps on $\bar G$ with respect to~$\bar c_\ell$. When $\ell = 0$, $F_\ell$ is the singleton $\{f\}$, where  $f(e) = 0$ for all $e\in \bar E$.

Given a flow $f\in F_\ell$, we define the subgraph of the original bipartite graph $G$ {\em induced by $f$} as follows:  \begin{align}\label{eq:defgraphinducedbysaturatedflow1}
G_f &:=(V_\alpha\cup V_\beta,E_f) \mbox{ with }\nonumber\\ 
E_f &:=\{ (\alpha_i,\beta_j)\in E \mid f(\alpha_i\beta_j) \neq 0\}.
\end{align}
In words, we select only edges of $G$ whose directed versions in $\bar G$ are used by the flow $f$.

Recall that for a flow $f\in \fl$, its value $|f|$ is given by Eq.~\eqref{eq:defvalueofflow}. 
We need the following definition:

\begin{definition}[Saturated flows]\label{def:saturatedflow}
Given the digraph $\bar G(n,m)$ and a nonnegative integer $\ell$, we say that a flow $f\in F_\ell$ on $\bar G(n,m)$ is {\em saturated} if $|f| = n\ell$. We denote by $\bfl$ the set of saturated flows on $\bar G(n,m)$.
\end{definition}

Note that a saturated flow is necessarily a max-flow because, by Def.~\ref{def:alg1}, the value of a flow $f$ cannot  exceed $n\ell$.  
Further, note that by the integrality theorem (see Sec.~\ref{ssec:background}), if a real-valued flow $f$ with $|f|=n\ell$ exists, then $\bar F_\ell$ is non-empty. 
The following lemma shows that saturated flows are in correspondence with disjoint left-perfect matchings.

\begin{lemma}\label{lem:satflowsmatchingscorresp}
Let $\ell \geq 1$ and $f \in F_\ell$. Then, $f \in \bar F_\ell$ if and only if $G_f$ is a union of $\ell$ disjoint left-perfect matchings.
\end{lemma}

\begin{proof}
Assuming that $f \in F_\ell$ and $G_f=(V_\alpha \cup V_\beta, E_f)$ is a disjoint union of $\ell$ left-perfect matchings, we show that $f\in \bar F_\ell$. 
It should be clear that $\deg(\alpha_i;G_f)=\ell$ for each $\alpha_i  \in V_\alpha$. Hence, there exist $1\leq i_1,\ldots, i_\ell\leq \ell$ so that $(\alpha_i,\beta_{i_j}) \in E_f$, for $1 \leq i \leq n$ and $1 \leq j \leq \ell$. From the definition of $G_f$ and the fact that $f$ is integer-valued, we have that $f(\alpha_i\beta_{i_j})\geq 1$. On one hand, using Eqns.~\eqref{eq:flowcond} and~\eqref{eq:defvalueofflow}, we obtain that  $$|f|=\sum_{i=1}^nf(s\alpha_i)=\sum_{(\alpha_i,\beta_{i_j})\in E_f}f(\alpha_i\beta_{i_j})\geq n \ell.$$ 
On the other hand, since $f \in F_\ell$, $|f|\leq n\ell$. 
Thus, we must have that $|f|=n\ell$ and, hence, $f \in \bar F_\ell$.  

Reciprocally, assuming that $f \in \bar F_\ell$, we show that $G_f$ is a disjoint union of $\ell$ left-perfect matchings.   
To this end, we claim that the degree of each left-node in $G_f$ is exactly $\ell$ and the degree of each right-node in $G_f$ is less than or equal to $\ell$. If this holds, then the result is an immediate consequence of Theorem~\ref{thm:exactkmatching}. 
We now prove the claim. For the left-nodes, because $f$ is saturated, $n\ell=|f|=\sum^n_{i = 1}f(s\alpha_i)$. 
By the definition of the capacity function~\eqref{eq:capacity},  $f(s\alpha_i) \leq \ell$. It follows that $f(s\alpha_i) = \ell$ for all $i = 1,\ldots,n$. 
Next, by the balance condition and the definition of $G_f$ in~\eqref{eq:defgraphinducedbysaturatedflow1}, 
$$f(s\alpha_i)=\sum_{j:(\alpha_i,\beta_{j})\in E}f(\alpha_i\beta_{j})=\sum_{j:(\alpha_i,\beta_j)\in E_f}f(\alpha_i\beta_j).$$ 
Further, by the capacity function~\eqref{eq:capacity} and the fact that $f$ is integer-valued, we have that  $f(\alpha_i\beta_{j}) = 1$ for $(\alpha_i,\beta_{j})\in E_f$ and, thus, there are exactly $n$ edges in $E_f$ incident to $\alpha_i$. Finally, for the right-nodes, one can apply similar arguments: first, from the capacity function, we have that $f(\beta_jt) \leq \ell$; then, the balance condition $f(\beta_jt)=\sum_{i:(\alpha_i,\beta_{j})\in E_f}f(\alpha_i\beta_{j})$ implies that $\deg(\beta_j;G_f)\leq \ell$. This proves the claim.
\end{proof}

With the above preliminaries, we now provide a solution to Problem~P1:

\begin{theorem}\label{th:maxflowgood}
Let $G(n,m)$ be a bipartite graph with $m \ge n$. Let $\bar G(n,m)$ be the digraph from Def.~\ref{def:alg1} and $\bar F_\ell$ be given as in Def.~\ref{def:saturatedflow}. Let 
$\ell^*:=\max \left \{\ell \geq 0 \mid \bfl \neq \varnothing \right \}$. 
Then, $\ell^*\leq m$ and the following hold:
\begin{enumerate}
    \item If $\ell^*\geq 0$, then  $\bar F_{\ell} \neq \varnothing$ for $0\leq \ell\leq \ell^*$ and for any $f \in \bar F_{\ell}$, the bipartite graph $G_f(n,m)$ given in~\eqref{eq:defgraphinducedbysaturatedflow1} is a union of $\ell$ disjoint left-perfect matchings. 
    \item The degree of strong resilience of $G(n,m)$ is $(\ell^*-1)$. 
\end{enumerate}
\end{theorem}

\begin{proof}
We first show that $\ell^*\leq m$. Suppose to the contrary that $\ell^*>m$; then, by Lemma~\ref{lem:satflowsmatchingscorresp}, $G$ contains at least $(m+1)$ disjoint left-perfect matchings, which contradicts item~2 of Cor.~\ref{cor:littlebc} saying that the degree of strong-resilience of $G(n,m)$ is at most $(m-1)$ (and, hence, $G(n,m)$ can  contain at most $m$ disjoint left-perfect matchings). We next establish the two conditions of the theorem. 
\vspace{.1cm}

\noindent
{\em Proof of item~1.} Because $\bar F_{\ell^*}$ is nonempty, by Lemma~\ref{lem:satflowsmatchingscorresp}, $G$ contains $\ell^*$ disjoint left-perfect matchings, denoted by $P_1,\ldots, P_{\ell^*}$. 
Let $G' = (V_\alpha\cup V_\beta, E')$ be the subgraph of $G$ induced by $P_1,\cdots, P_\ell$. We use $G'$ to define a flow $f$ on $\bar G$ as follows: 
$$
f(e) := 
\begin{cases}
\ell & \mbox{if } e=s\alpha_i\mbox{ for }\alpha_i \in V_\alpha,\\
\deg(\beta_j;G') & \mbox{if } e= \beta_jt, \\
1 & \mbox{if } e=\alpha_i\beta_j \mbox{ for }(\alpha_i,\beta_j) \in E',\\
0 & \mbox{otherwise}.  
\end{cases}
$$
It should be clear that $f\in \bar F_\ell$. 
Using again Lemma~\ref{lem:satflowsmatchingscorresp}, we have that for any $f\in \bar F_\ell$, $G_f$ is a union of $\ell$ disjoint left-perfect matchings. 
\vspace{.1cm}

\noindent
{\em Proof of item~2.} First, we consider the case $\ell^*= 0$ and show that $G$ does not have a left-perfect matching (i.e., $\srs G = -1$).   
  Suppose to the contrary that there exists a left-perfect matching $P$ in $G$; then, consider the flow $f$ on $\bar G$ defined as follows:  
$$f(e) := \begin{cases}
0 & \mbox{if } e=\alpha_i\beta_j \mbox{ with } (\alpha_i,\beta_j) \notin P, \\
1 & \mbox{otherwise}.
\end{cases}
$$
It is not hard to see that $f$ is a flow on $\bar G$ with respect to $\bar c_1$ (i.e., $\ell=1$). By construction, $f$ is a saturated flow in $\bar F_1$, which is a contradiction.  
For the other case where $\ell^*\geq 1$, 
the item follows from the definition of $\ell^*$ and Lemma~\ref{lem:satflowsmatchingscorresp}. 
\end{proof}

\begin{example}\normalfont
Consider the bipartite graph $G(4,5)$ given in Fig.~\ref{fig:algorithm1}. We run, e.g., the Ford-Fulkerson algorithm for the weighted digraph $\bar G(4,5)$ with $\ell = 1,2$. We show in Fig.~\ref{fig:fish2} the corresponding saturated flows, which implies that $\bar F_\ell \neq \varnothing$ for $\ell = 1,2$. Also, note that $\bar F_3 = \varnothing$ because the degrees of nodes $\alpha_1$ and $\alpha_4$ in $G(4,5)$ are both $2$. Using Theorem~\ref{th:maxflowgood}, we have that $G(4,5)$ is strongly $1$-resilient. \qed
\end{example}

\begin{figure}
    \centering
     \subfloat[]{
\includegraphics{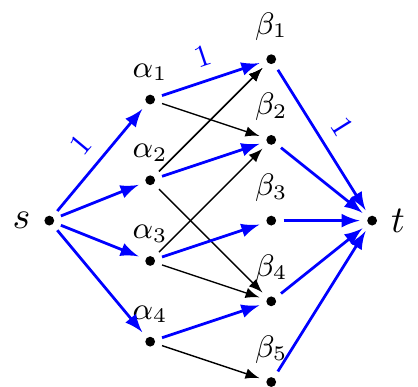}
}\qquad\qquad
    \subfloat[]{
  \includegraphics{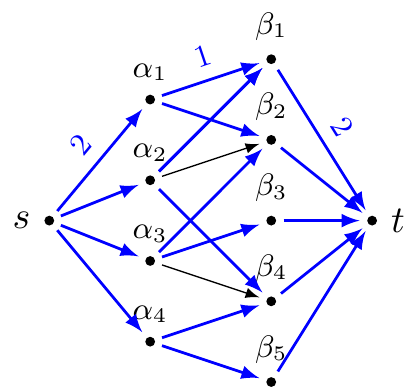}
}
    \caption{We show saturated flows $f_\ell$ on the graph $\bar G(4,5)$ as depicted in Fig.~\ref{fig:algorithm1} for $\ell = 1$ in (a) and $\ell = 2$ in (b). The edges with nonzero values under $f_\ell$ are highlighted in blue. In the case $\ell=2$, the two disjoint left-perfect matchings are  $\{(\alpha_1,\beta_1), (\alpha_2,\beta_4), (\alpha_3,\beta_2),(\alpha_4,\beta_5) \}$ and $\{(\alpha_1,\beta_2), (\alpha_2,\beta_1), (\alpha_3,\beta_3),(\alpha_4,\beta_4) \}$. Note, in particular, that the perfect matching of case $\ell=1$ is not part of the disjoint matchings for $\ell=2$.}
    \label{fig:fish2}
\end{figure}

The above theorem provides an algorithmic solution, of polynomial-time complexity, to P1, i.e. to determine $\srs G(n,m)$. The algorithm is as follows: Start by setting $\ell:=m$, and repeat the following procedure: 
\begin{enumerate}
    \item Construct the digraph $\bar G(n,m)$ and the capacity function $\bar c_\ell$. The complexity is $O(m+n)$.
    \item Run the Ford-Fulkerson algorithm on $\bar G(n,m)$ initialized at the zero flow and denote its output by $f$. The complexity is $O(n^2m\ell)$~\cite{cormen2009introduction}. If $|f|= n\ell$, then  return $\srs G(n,m) =\ell-1$. The algorithm is over. 
    \item If $|f| < n\ell$ and if $\ell \geq 2$ decrease the value of $\ell$ by $1$ and return to step~1. Otherwise, return $\srs G(n,m) =-1$ and the algorithm is over.
\end{enumerate}

\subsection{Minimal number of edges to increase $\srs G$}
In this subsection, we address the following simple question: Given a graph $G(n,m)$ which is a union of $k$ disjoint left-perfect matchings with $k \le m-1$, how many edges need to be added to this graph to obtain a bipartite graph  $G^*(n,m)$ which is a union of $(k+1)$ disjoint left-perfect matchings? Understanding this problem provides a solution to Problems P2 and, partially, to  P3 for the special case where $G$ is a union of disjoint left-perfect matchings. The advantage of the solution proposed here, when compared to the algorithms provided in the next subsection for solving general cases, is that it allows to establish an analytical bound on the number of edges needed to increase the degree of strong resilience.

To proceed, we introduce the natural notion of the {\em graph complement}. 
Given the {\em complete bipartite graph} $K = (V_\alpha\cup V_\beta,E_K)$ and the bipartite  graph $G=(V_\alpha\cup V_\beta, E)$, we denote by $G^c$ the complement of $G$ (in $K$); more precisely, $$G^c:=(V_\alpha\cup V_\beta, E_K-E).$$
\vspace{.1cm}
\noindent
{\em Special case $m=n$.} 
We have the following result:

\begin{lemma}\label{lem:specialcasen=m}
Let $G(n,n)$ be a union of $k$ disjoint perfect matchings, with $k \le n$. Let $G^c(n,n)$ be the complement of $G(n,n)$. Then, $G^c(n,n)$ is a union of $(n-k)$ disjoint perfect matchings. 
\end{lemma}

\begin{proof}
Since $m = n$, by Theorem~\ref{thm:exactkmatching}, the degree of every node in $G(n,n)$ is $k$. It then follows that the degree of each node in $G^c(n,n)$ is $(n-k)$. Using Theorem~\ref{thm:exactkmatching} again, we conclude that $G^c(n,n)$ is a union of $(n-k)$ disjoint perfect matchings. 
\end{proof}

It should be clear that adding {\em any} perfect matching of $G^c(n,n)$  to $G(n,n)$ yields a graph $G^*(n,n)$ which is  a union of $(k+1)$ disjoint perfect matchings. 
However, such a fact cannot be extended to the case  $m > n$ as seen in the following example:  

\begin{example}\normalfont
To see this, consider the graph $G(2,3)$ in Fig.~\ref{fig:exg23}, which depicts a simple case for which $m>n$. Here, $G(2,3)$ is the union of two disjoint left-perfect matchings. It is easy to see that $G^c(2,3)$ does {\em not} contain a  left-perfect matching.
Nevertheless, adding $G^c(2,3)$ to $G(2,3)$ still yields the graph $K(2,3)$ which is a disjoint union of $3$ left-perfect matchings. 
The key difference between this case and the one with $n=m$ is that in the latter case, one can always produce a graph $G^*(n,n)$ which is composed of the {\em all of the existing} $k$ disjoint perfect matchings of $G(n,n)$ and an {\em additional} disjoint perfect matching. 
In this example $G(2,3)$, the $3$ disjoint left-perfect matchings of $G^*(2,3)$ do not contain all of the left-perfect matchings that were used to express $G(2,3)$ as a disjoint union of perfect matchings. 
Generally speaking, this fact precludes the use of simple inductive arguments that rely on adding $n$ edges while keeping
 the $k$ disjoint perfect matchings that made $G(n,m)$.
\begin{figure}
    \centering
   \subfloat[]{
 \includegraphics{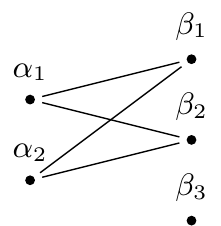}
}\qquad  
\subfloat[]{
 \includegraphics{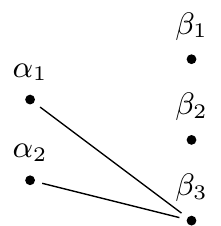}
}  \qquad
 \subfloat[]{
 \includegraphics{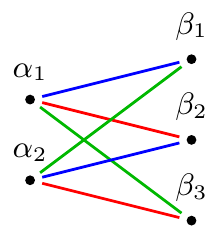}
}
    \caption{In (a), we depict a graph $G(2,3)$ which is the union of two disjoint left-perfect matchings  $P_1 = \{(\alpha_1,\beta_1), (\alpha_2,\beta_2)\}$   and $P_2= \{(\alpha_1,\beta_2), (\alpha_2,\beta_1) \}$. 
    In (b), we depict the complementary graph $G^c(2,3)$; it does not contain a left-perfect matching. In (c), we plot the three disjoint left-perfect matchings in the complete bipartite graph $K(2,3)=G(2,3) \cup G^c(2,3) $.
    }
    \label{fig:exg23}
\end{figure} 
\end{example}

\noindent
{\em General case $m \ge n$.} We establish the following  result, the proof of which will be {\em constructive}. 

\begin{theorem}\label{thm:ktok+1}
Let $G(n,m)$, with $n\leq m$, be a union of $k$ disjoint left-perfect matchings, for $k < m$. Then, one can add $\ell n$ edges, for $1 \le \ell \le m-k $,  to $G(n,m)$ such that the resulting graph $G^*(n,m)$ is a union of $(k+\ell)$ disjoint left-perfect matchings. 
\end{theorem}

The next result is then an immediate consequence of Theorem~\ref{thm:ktok+1}:  

\begin{corollary}
Given a strongly $k$-resilient $G(n,m)$, with $k < m$, and 
given a budget of $p$ additional edges, one can select $p$ edges $\{e_1,\ldots, e_p\}$ out of $G^c(n,m)$ such that the new graph $G(n,m) \cup \{e_1,\ldots,e_p\}$ is at least strongly $(k + \lfloor \frac{p}{n}\rfloor)$-resilient.
\end{corollary}

The remainder of the subsection is devoted to the proof of Theorem~\ref{thm:ktok+1}. It suffices to prove the Theorem for the case $\ell = 1$; one can then iteratively apply this case to  prove the general result.  
The proof has two parts: The first part relates the feasibility of the addition problem (i.e., the problem of adding $n$ edges to $G(n,m)$ to form a union of $(k+1)$ disjoint left-perfect matchings) to a max-flow problem; this is akin to what was done in Sec.~\ref{ssec:maxflow}. Here, we define a max-flow problem whose  capacity function allows us to decide whether the addition problem is feasible. Then, in the second part, relying on the max-flow min-cut Theorem, we compute explicitly the maximal capacity by computing the  corresponding minimal cut.    
\vspace{.1cm}

\noindent
{\em Max-flow formulation:}  We start by constructing another directed version of the bipartite graph $\hat G(n,m)$ with an appropriate capacity function.  
The solution of a newly defined max-flow problem on this graph will yield the edges  needed to increase the resilience:

\begin{definition}\label{def:alg2}
Given a bipartite graph $G = (V_\alpha\cup V_\beta, E)$ and an integer $k$, define the digraph $\hat G^c=(\hat V, \hat E)$ and the capacity function $\hat c_k: \hat E\to \Z_{\geq 0}$ as follows:
\begin{enumerate}
    \item Add two new nodes to $G$, denoted by $s$ and $t$: $$\hat V := V_\alpha \cup V_\beta \cup \{s, t\}.$$
    \item Create the edge set $\hat E$ as a union of $\hat E_0$ and $\hat E_1$ where  
\begin{equation}\label{eq:defE0E1cop}
\begin{aligned} 
\hat E_0 &:= \{s\alpha_i \mid \alpha_i\in V_\alpha\} \cup \{ \alpha_i\beta_j \mid (\alpha_i,\beta_j) \notin E\}, \\ 
\hat E_1 &:= \{\beta_j t \mid \beta_j\in V_\beta\}.
\end{aligned}
\end{equation}

\item If $e\in \hat E_0$, then $\hat c_k(e) := 1$; if $e = \beta_j t\in \hat E_1$, then 
\begin{equation}\label{eq:capacity2}
    \hat c_k(e):=\max\{0,(k+1) - \deg(\beta_j;G)\}.
\end{equation}
\end{enumerate}
\end{definition}

We illustrate the definition in Fig.~\ref{fig:algorithm2}. 

\begin{figure}
    \centering
    \subfloat[]{
 \includegraphics{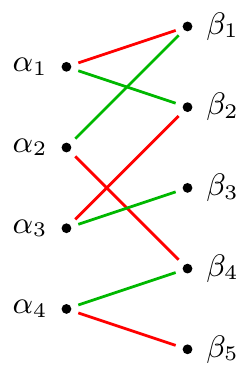}
} \qquad \qquad
\subfloat[]{
  \includegraphics{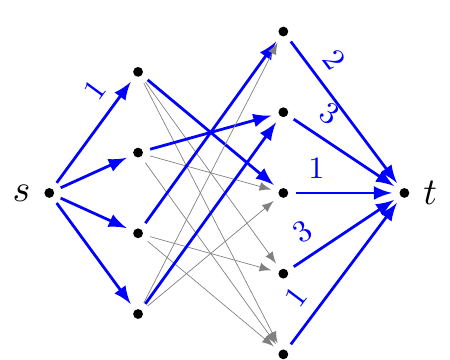}
}
\caption{The bipartite graph $G(4,5)$ in (a) is a union of two disjoint left-perfect matchings, highlighted in red and green. The weighted digraph in (b) is $\hat G^c(4,5)$ for $k=3$. Correspondingly, by Def.~\ref{def:alg2}, the capacity of an edge $\beta_j t$ is  given by $4 - \deg(\beta_j; G)$, whereas the capacity of each remaining edge is $1$. We then run the Ford-Fulkerson algorithm and highlight (in blue) a solution $f$ to the max-flow problem~\eqref{eq:defmaxflowalgo2}. By Prop.~\ref{prop:opapa}, the capacity of any such solution is given by $|f| = n = 4$. }
    \label{fig:algorithm2}
\end{figure}

Let $\hat F_k$ be the set of integer-valued flow maps on $\hat G^c(n,m)$. We will now relate the max-flow problem on $\hat G^c(n,m)$:  
\begin{equation}\label{eq:defmaxflowalgo2}
    \max_{f\in \hat F_k} |f|
\end{equation}
to Theorem~\ref{thm:ktok+1}. As mentioned above, requiring an integer solution is not constraining; it suffices to use the Ford-Fulkerson algorithm.

\begin{proposition}\label{prop:GtoG'kk1}
Let $G(n,m)$ be a union of $k$ disjoint left-perfect matchings, for $0\leq k < m$, and $\hat G^c(n,m)$ be given in Def.~\ref{def:alg2}.   
Let $f$ be a solution to problem~\eqref{eq:defmaxflowalgo2}. If $|f| = n$, then there exist $n$ edges   $\{e_1,\ldots, e_n\} \in G^c(n,m)$ so that $G^*(n,m):=G(n,m) \cup \{e_1,\ldots,e_n\}$  is  a union of $(k+1)$ disjoint left-perfect matchings.
\end{proposition}

\begin{proof}
Given the flow $f$ on $\hat G^c$, we let $G^c_f$ be the subgraph of $G^c$ induced by~$f$ as defined in~\eqref{eq:defgraphinducedbysaturatedflow1}. 

Because $|f| = n$ and because the capacity assigned to the edges $s\alpha_i$, for $\alpha_i\in V_\alpha$, is~$1$, the inflow at every node $\alpha_i$ is also~$1$. Also, since the capacities of edges of type $\alpha_i\beta_j$ are  $1$, we have  that there are exactly $n$ edges of this type for which $f$ is nonzero, and thus there are exactly $n$ edges in $G^c_f$. 
By construction, these edges are incident to $n$ {\em distinct} left nodes (as otherwise, it implies that an edge of type $s\alpha_i$ has a flow above its capacity of $1$). Denote by  $\{e_1,\ldots, e_n\}$ this set of edges in $G^c_f$. 

We show  that adding this set of edges to $G$  yields a $G^*$ which is a union of $(k+1)$ disjoint left-perfect matchings. We do so by verifying that $G^*$ satisfies the two items in Theorem~\ref{thm:exactkmatching}:
\begin{enumerate}
    \item $\deg(\alpha_i;G^*)=k+1$, for all $\alpha_i \in V_\alpha$.   This holds because of the following three facts: 
    First, by assumption, $\deg(\alpha_i, G) = k$. Next, note that $G$ and $G^c_f$ have disjoint sets of edges. Finally, the edges $e_1,\ldots,e_n$ in $G^c_f$ are incident to $n$ distinct left nodes.
    
    \item $\deg(\beta_j;G^*)\leq k+1$ for all $\beta_j \in V_\beta$. This holds because of the following three facts: First, by assumption, $\deg(\beta_j;G)\le k $ for $\beta_j \in V_\beta$. Second, recalling that the  capacities of the edges  in $\hat E_1$ are given in Eq.~\eqref{eq:capacity2}, we have that for each right node $\beta_j$, 
$$
\deg(\beta_j; G^c_f) \le (k+1) - \deg(\beta_j;G).
$$
Finally, because $G^*$ is the disjoint union of $G$ and $G^c_f$,  $\deg(\beta_j; G^*) = \deg(\beta_j;G) + \deg(\beta_j;G^c_f)\leq (k+1)$. 
\end{enumerate}
We have thus shown that the two items of Theorem~\ref{thm:exactkmatching} are satisfied by $G^*$. This completes the proof.
\end{proof}

Equipped with the above Proposition,  Theorem~\ref{thm:ktok+1} is easily seen to be equivalent to the following result:

\begin{proposition}\label{prop:opapa}
Let $G(n,m)$ be a union of $k$ disjoint left-perfect matchings, for $0\leq k < m$, and $\hat G^c(n,m)$ be as in Def.~\ref{def:alg2}.  
Let $f$ be a solution to the max-flow problem~\eqref{eq:defmaxflowalgo2}. Then, $|f| = n$.
\end{proposition}

\begin{proof}
To prove the result, we rely on the use of the max-flow min-cut Theorem (see Lemma~\ref{lem:maxflowmincut}), which applied here reduces the problem to showing that for every cut $(S,T)$ in  $\hat G^c$, its capacity $c(S,T) \geq n$ and, furthermore, this  lower-bound is realizable. 

For a given cut $(S,T)$ in $\hat G^c$, we let $S_\alpha:=S \cap V_\alpha$ and $T_\alpha:=T \cap V_\alpha$ be the sets of left-nodes contained in $S$ and $T$ respectively. Similarly, we define $S_\beta:=S \cap V_\beta$ and $T_\beta:=T\cap V_\beta$. 
Let $p:= |T_\alpha|$ and $q:= |S_\beta|$. For every such cut, we can write its capacity into the sum of three terms:  
\begin{equation}\label{eq:capcut}
    c(S,T) = c(s, T_\alpha) + c(S_\alpha, T_\beta) + c(S_\beta,t),
\end{equation}
where the three terms are given by
\begin{equation}\label{eq:threecuts}
\left\{
\begin{array}{lll}
 c(s, T_\alpha)& := &\sum_{\alpha_i \in T_\alpha} 
c(s\alpha_i), \\
 c(S_\alpha, T_\beta)&:= &\sum_{\alpha_i \in S_\alpha, \beta_j\in T_\beta} c(\alpha_i\beta_j), \\
 c(S_\beta, t)& := &\sum_{\beta_j \in S_\beta} 
c(\beta_j t).
\end{array}
\right.
\end{equation}

\begin{figure}
    \centering
 \includegraphics{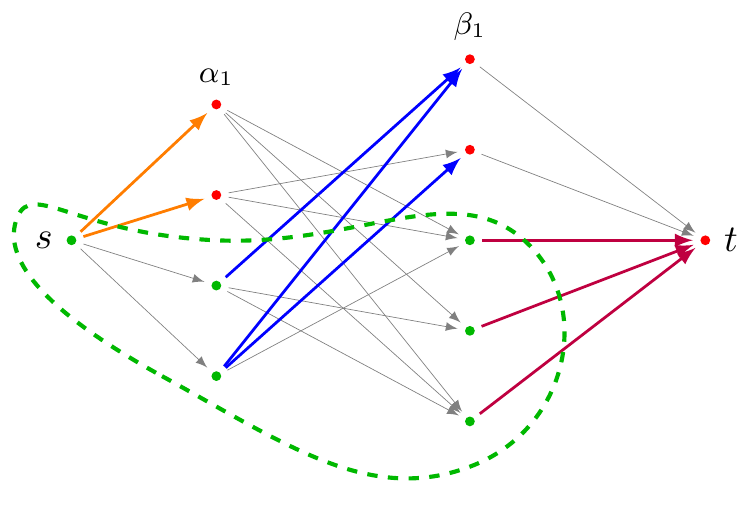}

\caption{We illustrate the three terms defined in~\eqref{eq:threecuts}. In this digraph, we let the cut $(S, T)$ be such that the nodes depicted in green (resp. red) are nodes in $S$ (resp. $T$). The set $S$ is circled by the dashed green line. Then, the term $c(s,T_\alpha)$ is the sum of the capacities of the edges depicted in orange, the term $c(S_\alpha, T_\beta)$  is the sum of the capacities of the edges depicted in blue, and the term $c(S_\beta, t)$ is the sum of the capacities of the edges depicted in purple. The capacity of the cut $(S,T)$ is easily seen to be sum of these three terms.}
    \label{fig:pear}
\end{figure}
We evaluate below these three terms (also, see Fig.~\ref{fig:pear} for an illustration): 
\vspace{.1cm}

\noindent
{\em First term $c(s, T_\alpha)$.}  Note that by item~3 of Def.~\ref{def:alg2}, $\hat c_k(s \alpha_i) = 1$, for $\alpha_i \in T_\alpha$, so 
    \begin{equation}\label{eq:opapafirst}
    c(s, T_\alpha)  = p.
    \end{equation}
    
\vspace{.1cm}

\noindent
{\em Second term $c(S_\alpha, T_\beta)$.} We first establish the following inequality:
    \begin{equation}\label{eq:lowerboundfor2ndterm}
    c(S_\alpha, T_\beta) \ge \sum_{\alpha_i \in S_\alpha} \deg(\alpha_i; G^c) - \sum_{\beta_j\in S_\beta} \deg(\beta_j; G^c).
    \end{equation}
    To see it holds, first note that total number of out-going edges incident to the nodes $\alpha_i \in S_\alpha$ is exactly given by $\sum_{\alpha_i\in S_\alpha}\deg(\alpha_i; G^c)$. Every such outgoing edge is necessarily incident to either a node in $S_\beta$ or a node in $T_\beta$. Furthermore, the  number of incoming edges incident to nodes $\beta_j\in S_\beta$ is given by $\sum_{\beta_j\in S_\beta} \deg(\beta_j; G^c)$. Similarly, every such incoming edges can be incident to either a node in $S_\alpha$ or a node in $T_\alpha$.  
    It then follows that the number of  edges incident to both $S_\alpha$ and $T_\beta$ in $\hat G^c$ is bounded below by the expression on the right hand side  of~\eqref{eq:lowerboundfor2ndterm}. Because $\hat c_k(\alpha_i\beta_j)=1$, the inequality~\eqref{eq:lowerboundfor2ndterm} holds. 
    
    Now, we evaluate the two sums on the right hand side of~\eqref{eq:lowerboundfor2ndterm}.  
    For the first sum, since $G$ is a union of $k$ disjoint left-perfect matchings, we have that $\deg(\alpha_i;G^c)= (m-k)$ for all $i=1,\ldots, n$. Further, since $|S_\alpha|=n-|T_\alpha|=n-p$,       \begin{equation}\label{eq:firstterm}
        \sum_{\alpha_i \in S_\alpha} \deg(\alpha_i; G^c)  = (m - k)(n-p). 
    \end{equation}    
    For the second sum, 
    since the degree of each node $\beta_j$ in $G^c$ is $n - \deg(\beta_j;G)$ and since $|S_\beta|=q$,  
    \begin{equation}\label{eq:secondterm}    
        \sum_{\beta_j\in S_\beta} \deg(\beta_j; G^c)  = qn - \sum_{\beta_j\in S_\beta} \deg (\beta_j; G).
    \end{equation}
    
     Plugging Eqs.~\eqref{eq:firstterm} and~\eqref{eq:secondterm} in~\eqref{eq:lowerboundfor2ndterm}, we obtain that
    \begin{equation}\label{eq:opapa2nd}
        c(S_\alpha, T_\beta) \ge (m-k)(n-p) - qn + \sum_{\beta_j\in S_\beta} \deg(\beta_j;G).
    \end{equation}
    
\vspace{.1cm}

\noindent
{\em Third term $c(S_\beta,t)$.} From~\eqref{eq:capacity2} and the fact that $|S_\beta|=q$, 
\begin{equation}\label{eq:opapa3rd} 
c(S_\beta, t) = (k+1)q - \sum_{\beta_j\in S_\beta} \deg(\beta_j;G).
\end{equation}

We now use the facts just established to show that $c(S, T) \ge n$. Specifically, we use Equations~\eqref{eq:capcut}-\eqref{eq:opapafirst} and~\eqref{eq:opapa2nd}-\eqref{eq:opapa3rd} to obtain that
    \begin{align*}
    c(S,T) &\ge p + (m-k)(n-p) - qn + (k+1)q \\
    & \ge (p+q) + (m-k)(n-p) - (n-k)q \\
    & \ge (p+q) + (m-k)(n-p-q) \\
    & \ge (p+q) + (n - p - q) \\
    & \ge n.
    \end{align*}
To obtain the second line from the first, we simply rearrange terms. To obtain the third line from the second, we use the fact that $m \geq n$, and using furthermore the assumption that $m > k$, we obtain the fourth line from the third. This concludes the proof.
\end{proof}

At the end of this subsection, we conclude that 
given a graph $G(n,m)$ which is a union of $k$ disjoint left-perfect matchings, and $1 \leq \ell \leq m-k$, one can obtain a 
$G^*(n,m) \succ G(n,m)$ which is a union of $(k+\ell)$ disjoint left-perfect matchings using the following algorithm in polynomial-time: Start by setting $\ell'=0$ and $G'=G$; While $\ell' <\ell$, repeat the following steps
\begin{enumerate}
    \item Construct $\hat G'^c$ and $\hat c_{k+\ell'}$ given by Def.~\ref{def:alg2}. The complexity is $O(m+n)$. 
    \item Run the Ford-Fulkerson algorithm on $\hat G'^c$ and denote by $f$ the output. The complexity is $O(n^2(m-k))$. 
    \item Update $G'$ to be the union of the current $G'$ and $G^c_f$ (note that $G'$ and $G^c_f$ are edge-wise disjoint) and increase $\ell'$ by $1$. The complexity is $O(n)$.
\end{enumerate}

\subsection{Solutions to Problems P2 and P3}

In this section, we let $G(n,m)$ be an arbitrary bipartite graph, with $m \ge n$ as above, and $G^c(n,m)$ be its complement in the complete bipartite graph $K(n,m)$.
 
Recall that for Problem P2, we aim to find a set of edges in $G^c$ of least cardinality which, when added to $G$, yields a graph which is (strongly) $k$-resilient, 
and for Problem P3,  given a budget of $p$ edges and a graph $G$, we aim to maximize the degree of (strong) resilience by optimally choosing these $p$ additional edges.

We provide below complete solutions to the two problems for {\em strong} resilience, together with a polynomial-time algorithm that fulfills the respective goals.

\vspace{.1cm}

\noindent
{\em Fair matchings and fair $b$-matchings.}  
One of the major hurdles in adding edges to $G$ to increase the number of left-perfect matchings is that one has the option to use edges that already exist in $G$ to create said additional matchings. The use of these existing edges should of course be prioritized as much as possible over the addition of new edges. We can recast this problem by considering  the embedding of $G$ into the complete bipartite graph $K=(V_\alpha,V_\beta,E_K)$. This embedding allows us to view both Problems P2 and P3, which are dual to each other, as the problem of selecting edges in $K$ to obtain a desired number of disjoint left-perfect matchings {\em while maximizing} the use of edges that belong to $G$. 
Moreover, this point of view will allow us to appeal to algorithms that obtain such matchings, and thus solve the above-mentioned problem, in polynomial time.

To proceed, we rely on  the notion of fair matching and, more specifically, {\em fair  b-matching} in a bipartite graph. Such matchings are described in relation to the following additional structures on a graph:

\begin{enumerate}
    \item A {\em capacity function} $\mu$ at the {\em nodes}, which is a positive-integer valued function $\mu:V_\alpha \cup V_\beta \to \Z_{\geq 0}$ which provides an upper bound on the degrees of the nodes in a $b$-matching.
    \item  A {\em priority order} for the possible neighbors of each node. Assuming that there are $r$ different priorities, we label them as $1,\ldots,r$. The priority order indicates which edges of $G$ are preferred to appear in the matching. 
    \end{enumerate}
For our purpose, we only need to consider a particular class of fair $b$-matching problems: (1) Elements of that class are defined on the complete bipartite graph $K$; (2) The capacity functions $\mu$ are constant functions with value equal to $(k^*+1)$, where $k^*$ is the target degree of strong-resilience; and (3) The priority order has $r = 2$ classes, and is induced by $G$ in the sense that a node $\alpha_i$ (resp. $\beta_j$) prefers $\beta_j$ (resp. $\alpha_i$) if $(\alpha_i,\beta_j)$ is an edge in $G$. We refer the reader to~\cite{anstee1987polynomial,huang2016fair}) for a general introduction to $b$-matchings. 

Formally, we introduce the following definition of $b$-matching and fair $b$-matching considered in this paper: 

\begin{definition}[$b$-matching and fair $b$-matching]\label{def:bmatching} Let $K = (V_\alpha \cup V_\beta, E_K)$ be the complete bipartite graph and $G = (V_\alpha \cup V_\beta, E)$ be a subgraph of $K$. 
A {\em $b$-matching} is a subset $P \subseteq E_K$ for which each vertex $v\in V_\alpha \cup V_\beta$ is incident to at most $(k^*+1)$ edges of $P$. 
A {\em fair b-matching} is a $b$-matching of {\em maximal cardinality} so that $|P \cap E|$ is maximized.
\end{definition}

We make the following observation:

\begin{lemma}\label{lem:bmatching}
If $P$ is a $b$-matching with of maximal cardinality, then $P$ is a disjoint union of $(k^*+1)$ disjoint left-perfect matchings, and vice- versa.  
\end{lemma}

\begin{proof}
First, it should be clear that if $P$ is a $b$-matching, then by the capacity condition in Def.~\ref{def:bmatching}, $|P|\leq (k^*+1)n$. 
Next, let $P$ be an arbitrary union of $(k^*+1)$ disjoint left-perfect matchings. Then, $P$ satisfies the capacity condition and $|P| = (k^*+1)n$. Thus, such a $P$ is a $b$-matching of maximal cardinality.

Now, let $P$ be a $b$-matching of maximal cardinality and $G=(V_\alpha\cup V_\beta,P)$. Suppose that $P$ is {\em not} a union of $(k^*+1)$ disjoint left-perfect matchings; then, by Theorem~\ref{thm:exactkmatching} and the capacity condition in Def.~\ref{def:bmatching}, there exists at least one node $\alpha_i \in V_\alpha$ such that 
\begin{equation}\label{eq:ourlittledeg}
\deg(\alpha_i;G) < k^*+1.
\end{equation}
To see this, note that a graph $G$ induced by a $b$-matching always satisfies item~2 of Theorem~\ref{thm:exactkmatching}; hence, if $G$ is not a union of $(k^*+1)$ disjoint left-perfect matchings, then item~1 cannot be met, which implies that $\deg(\alpha_i;G)< k^*+1$ for some $\alpha_i$. 
On the one hand, as a consequence of Eq.~\eqref{eq:ourlittledeg}, the cardinality of $P$ is strictly less than $(k^*+1)n$. 
On the other hand, by the arguments at the beginning of the proof,  
if we let $P'$ be an arbitrary union of $(k^*+1)$ disjoint left-perfect matchings, then $P'$ is a $b$-matching with  $|P'| = (k^*+1)n>|P|$, which is a contradiction.
\end{proof}

If $P^*\subseteq E_K$ is a $b$-matching of maximal cardinality, a {\em fair} $b$-matching can be obtained by first finding all $b$-matchings of cardinality $|P^*|$ and, amongst those, selecting one which maximizes  $|P^* \cap E|$.  
It is known that finding a fair $b$-matching in $K(n,m)$ can be done in polynomial time. 
To be more precise, if we let $N:=m+n$ be the number of nodes of $K(n,m)$ and $M:=mn$ be the number of edges in $K(n,m)$, there exist algorithms solving fair $b$-matching problems in $O(NM\log(N^2/M)\log (N))$ time, using $O(M)$ space~\cite{huang2016fair}. 

\vspace{.1cm}

\noindent
{\em Solution to Problem P2 for strong resilience.}
We  now reduce Problem P2 to the fair $b$-matching problem. 
Let $k^*$ be the target degree of strong resilience.  
If $\srs G \geq k^*$, then no additional edge is needed and we are done. Otherwise, we have the following result: 

\begin{theorem}\label{thm:solutiontop2}
Let $G(n,m) = (V_\alpha \cup V_\beta, E)$ be a bipartite graph with $m\ge n$ and $\srs G <k^*$ with $0\leq k^*\leq (m-1)$. Let $P^*$ be a solution to the fair $b$-matching problem of Def.~\ref{def:bmatching}. 
Then, the following hold: 
\begin{enumerate}
    \item The graph $G^*(n,m):=(V_\alpha \cup V_\beta, E \cup P^*)$ is strongly $k^*$-resilient.
    \item The minimal number of edges out of $G^c(n,m)$ one needs to add to $G(n,m)$ to obtain a strongly $k^*$-resilient graph $G^*(n,m)$ is given by
    \begin{equation}\label{eq:minimalnumberofedges}
\delta^* := |P^*| - |P^* \cap E|.
\end{equation} 
\end{enumerate}
\end{theorem}

Note that for a given graph $G(n,m)$, $\delta^*$ depends only on the number $k^*$ (in particular, it does  {\em not} depend on the choice of $P^*$ from Def.~\ref{def:bmatching}). If necessary, we will write explicitly $\delta^*(k)$ to indicate such dependence.    

By item~1, we have that  $G^*\succ G$;  by item~2, $G^*$ contains the least number of additional edges so as to be strongly $k^*$-resilient. Thus, Problem P2 is indeed solved for strong $k^*$-resilience.

\begin{proof}[Proof of Theorem~\ref{thm:solutiontop2}]
We establish the two items below:
\vspace{.1cm}

\noindent
{\em Proof of item~1.} We show that $G^*$ contains exactly $(k^*+1)$ disjoint left-perfect matchings. By Lemma~\ref{lem:bmatching}, $P^*$ is a union of $(k^*+1)$ disjoint left-perfect matchings. Since the edge set of $G^*$ contains $P^*$, $G^*$ contains at least $(k^*+1)$ disjoint left-perfect matchings.  
Now, suppose, to the contrary, that $G^*$ contains (exactly) $\bar k$ disjoint left-perfect matchings, with $\bar k \geq (k^*+2)$; then, we let $\{P_i\}_{i=1}^{\bar k}$ be a set of such matchings. Let $\rho_i:=|P_i - E|$, i.e., $\rho_i$ is the number of edges in $P_i$ but not in $E$. It follows that $|E\cap P^*| = |P^*| - \sum^{\bar k}_{i = 1}\rho_i$. 
Relabel the $P_i$, if necessary, so that $\rho_1\geq \cdots \geq \rho_{\bar k}$. Then, $\rho_1$ has to be positive, since otherwise $\rho_i = 0$ for all $i = 1,\ldots, \bar k$, which implies that all these $\bar k$ matchings $P_i$ are contained in $E$, contradicting the assumption that $\srs G < k^*$ (the same arguments imply that $\rho_2$ has to be positive as well). Now, let $P:= \cup^{k^*+2}_{i = 2}P_i$. Then, by Lemma~\ref{lem:bmatching}, $P$ is a $b$-matching and $|P|=|P^*|$.  Moreover, $|E\cap P| = |P|- \sum^{k^*+2}_{i = 2}\rho_i > |E\cap P^*|$, 
which contradicts the assumption that $P^*$ is a {\em fair} $b$-matching.
This proves item~1.
\vspace{.1cm}

\noindent
{\em Proof of item~2.} Let $G^* = (V_\alpha\cup V_\beta, P^*)$ be induced by an arbitrary {\em fair} $b$-matching $P^*$. Then, the cardinality $|P^* \cap E|$ is maximized over all $b$-matchings $P$ of maximal cardinality and, thus, $|P|-|P\cap E|$ is minimized.  This proves item~2 and completes the proof.
\end{proof}

\noindent
{\em Solution to Problem P3 for strong resilience.}
Since Problem P3 is dual to Problem P2, it can similarly be solved via a reduction to the fair $b$-matching problem. 
Precisely, we have the following result: 

\begin{theorem}\label{thm:solutiontop3}
Let $G(n,m) = (V_\alpha \cup V_\beta, E)$ be a bipartite graph with $m\ge n$ and $p$ be a positive integer. Then, the solution to the following optimization problem (Problem P3 for strong resilience):
\begin{align*}
\max \srs G^*(n,m)  = (V_\alpha \cup V_\beta, E^*), \\
\mbox{ s.t. } G^*(n,m) \succeq G(n,m) \mbox{ and } |E^*| - |E| = p
\end{align*}
is given by
$$
\max \{ k \mid \delta^*(k) \le p\}, 
$$
where $\delta^*(k)$ is defined in~\eqref{eq:minimalnumberofedges}. 
\end{theorem}

\begin{proof}
It is an immediate consequence of Theorem~\ref{thm:solutiontop2}: On the one hand, for any $k$ with $\delta^*(k) \leq p$, one can always add $p$ edges out of $G^c$ to $G$ so that the resulting graph $G^*$ is at least strongly $k$-resilient. On the other hand, it is clear from the definition of $\delta^*(k)$ that it is infeasible to obtain a graph with strong $k$-resilience by adding fewer than $\delta^*(k)$ edges to $G$. 
\end{proof}

\section{Conclusion}
We have addressed in this paper the resilience of the structural rank of sparsity patterns. The first step in our approach to solve the problems was to recast them as problems posed for bipartite graphs. We then provided a  characterization of bipartite graphs corresponding to sparsity patterns of full rank (see Theorem~\ref{thm:exactkmatching}). Based on this characterization, we provided provably correct, polynomial-time algorithms to solve three problems dealing with (strong) resilience of the pattern: Given a sparsity pattern, (1) what is its degree of (strong) resilience, i.e., how many $\star$-entries can be removed without affecting the structural rank; (2) what is the minimal number of $\star$-entries one needs to add to a pattern so as to reach a target degree of (strong) resilience; and (3) given that one can add $p$ $\star$-entries to a sparsity-pattern, where to place these entries so as to maximize the degree of (strong) resilience.

\bibliographystyle{ieeetr}
\bibliography{main}

\appendix

\section*{Proof of Lemma~\ref{lem:frpm}}
\begin{proof}
To prove the result, we first introduce a few preliminaries. 
Given a digraph $\vec G = (V, \vec E)$ on $n$ nodes, we say that $\vec G$ admits a {\em Hamiltonian decomposition}~\cite{BELABBAS2013981} if there is a subgraph $\vec G' = (V, \vec E')$, with $\vec E' \subseteq \vec E$, such that $\vec G'$ is a {\em disjoint} union of cycles.   
To a sparsity pattern $\cS(n,n)$, we can associate a digraph $\vec G = (V, \vec E)$ on $n$ nodes $\gamma_1,\ldots, \gamma_n$ as follows: $\gamma_i\gamma_j \in \vec E$ if the pair $(i,j)$ belongs to $E(\cS(n,n))$. 
It is well-known that $\cS(n,n)$ admits a matrix of full rank if and only if $\vec G$ admits a Hamiltonian decomposition.  Let $G(n,n)$ be the bipartite graph associated with the same sparsity pattern $\cS(n,n)$. 
It is also well know that $G(n,n)$ has a perfect matching if and only if the digraph $\vec G$ admits a Hamiltonian decomposition (see~\cite{belabbas_algorithmsparse_2013} for a simple account of this fact). 

With the above preliminaries, we now return to establish Lemma~\ref{lem:frpm}. 
First, note that the rank $\cS(n,m)$ is $n$ if and only if there exist $n$ columns so that the sub-pattern induced by these columns is of full rank; precisely, there exists $1\leq j_1 < \cdots <j_n \leq m$ so that the sparsity pattern $\cS'(n,n)$ defined by the index set $$E(\cS')=\{(i,j_k)\in E(\cS) \mid 1\le k\le n \}$$
is of rank~$n$. 
Owing to the preliminaries above, $\cS'(n,n)$ is of full rank if and only if the associated digraph $\vec G'$ on $n$ nodes admits a Hamiltonian decomposition. Furthermore, the existence of this Hamiltonian decomposition implies that the bipartite graph $G'(n,n)$ corresponding to $\cS'(n,n)$ contains a perfect matching. 

From the definition of $\cS'$, it is not hard to see that $G'(n,n)$ can be realized as a subgraph of $G(n,m)$; more precisely, $G'(n,n)$ is the subgraph of $G(n,m)$ induced by the nodes $\alpha_i \in V_\alpha$ and nodes $\beta_{j_1},\ldots,\beta_{j_n}$. Thus, a perfect matching in $G'(n,n)$ is mapped using the above inclusion to a left-perfect matching in $G(n,m)$.  
We thus conclude that the rank of $\cS(n,m)$ is $n$ if and only if $G$ has a left-perfect matching. \end{proof}

\end{document}